\documentclass[smallextended]{svjour3}       
\smartqed  
\usepackage{graphicx}
\usepackage{amssymb}
\usepackage{amsmath,dsfont}
\newtheorem{assumption}[theorem]{Assumption}

\newcommand*{\N}{\ensuremath{\mathbb{N}}}
\newcommand*{\R}{\ensuremath{\mathbb{R}}}
\newcommand*{\Z}{\ensuremath{\mathbb{Z}}}
\newcommand*{\C}{\ensuremath{\mathbb{C}}}
\newcommand{\A}{{\mathcal{A}}}
\newcommand{\UC}{\mathbb{S}^1}
\newcommand{\p}{{per}}
\renewcommand{\i}{\mathrm{i}}
\newcommand{\grad}{\nabla}
\newcommand{\K}{{\mathcal{K}}}
\newcommand{\B}{{\mathcal{B}}}
\newcommand{\T}{{\mathcal{T}}}
\renewcommand{\d}[1]{\,\mathrm{d}#1 \,}
\newcommand{\F}{\mathcal{F}} 
\newcommand{\I}{{\mathcal{I}}}

\usepackage{color}
\newcommand{\high}[1]{{\color{black}{#1}}}

\begin{document}

\title{Numerical methods for scattering problems in periodic waveguides
}

\titlerunning{Numerical methods periodic waveguides}        

\author{Ruming Zhang
}


\institute{R. Zhang \at
              Institute of Applied and Numerical mathematics, Karlsruhe Institute of Technology, Karlsruhe, Germany \\
              Tel.: +49-721-608-42062\\
              Fax: +49-721-608-43197\\
              \email{ruming.zhang@kit.edu}           
}

\date{Received: date / Accepted: date}

\maketitle

\begin{abstract}
In this paper, we propose new numerical methods for scattering problems in periodic waveguides. As well-posedness is not always guaranteed, we are looking for solutions obtained via the { Limiting Absorption Principle (LAP)}, which are called LAP solutions. The method is based on a newly established contour integral representation of the LAP solutions.  Based on the Floquet-Bloch transform and analytic Fredholm theory, when the wavenumber satisfies certain conditions, the LAP solution can be written as an integral of quasi-periodic solutions on a contour. The definition of the contour depends on both the wavenumber and the periodic structure.  Compared with other numerical methods, we do not need the LAP process during numerical approximations, thus a standard error estimation is easily carried out. Based on this method, we also develop a numerical solver for  halfguide problems. The method is based on the result that any LAP solution of a halfguide problem can be extended into the LAP solution of a fullguide problem. We first approximate the source term from the boundary data by a regularization method, and then the LAP solution is computed from the corresponding fullguide problem. The method is also extended to more general wavenumbers with an interpolation technique. At the end of this paper, we also give some numerical results to show the efficiency of our numerical methods.
\end{abstract}

\section{Introduction}
\label{intro}

Numerical simulation for scattering problems in periodic waveguides is an interesting topic in both mathematics and related technologies, due to its wide applications in optics, nanotechnology, etc. As the well-posedness of the scattering problems is not always true, the {\em Limiting Absorption Principle (LAP)} is commonly used to find out the ``physically meaningful'' solution. That is, the ``physically meaningful \high{solution}'' is assumed to be the limit of a family of solutions with absorbing material, when the positive absorption parameter tends to $0$. In this paper, the limit is called the {\em LAP solution}.  It has been  proved that LAP solutions exist for planar waveguides filled up with periodic material, and we refer to \cite{Hoang2011,Fliss2015,Kirsc2017,Kirsc2017a} for different proofs.

In recent years, some numerical methods have been developed to solve this kind of problems based on the LAP. 
As the problem with absorbing medium is uniquely solvable and the unique solution decays exponentially along the periodic waveguide, the solution is easily approximated by problems defined in bounded domains and then computed by standard numerical methods. We refer to \high{\cite{Ehrhardt2009,Yuan2006}} for a numerical method with the help of a so-called {\em recursive doubling process} to approximate Robin-to-Robin maps on left-and right-boundaries of a  periodic cell, and the original solution is approximated by the one with a sufficiently small positive absorption parameter. Based on this idea, the solution is also approximated by an extrapolation technique with respect to small positive absorption parameters in \cite{Ehrhardt2009a}, and the method is also  extended to  scattering problems in locally perturbed periodic layers in \cite{Sun2009}. On the other hand, the LAP has also been applied to pick up ``proper'' propagation modes. With the asymptotic behaviour of propagating modes when the absorption parameter tends to $0$, it is possible to determine whether certain modes propagate to the left or to the right. With these propagating modes, Dirichlet-to-Neumann maps on the boundaries of periodic cells are approximated and the solution for the whole problem is computed. For details we refer to \cite{Joly2006,Fliss2009a}. For other numerical methods we also refer to \cite{Lecam2007,Alcaz2013,Dohna2018}.

Recently, the Floquet-Bloch transform has been applied to both theoretical analysis and numerical simulation for scattering problems in periodic structures. We refer to \cite{Fliss2012,Coatl2012,Hadda2016} for scattering problems with (perturbed) periodic media, to \cite{Lechl2016,Lechl2016a,Lechl2017,Zhang2017e} for problems with periodic surfaces, and to \cite{Fliss2015,Kirsc2017,Kirsc2017a} for periodic waveguides. From the Floquet-Bloch theory, the unique solution of the periodic waveguide problem with absorption is written as a contour integral on the unit circle, where the integrand is a family of quasi-periodic solutions depending analytically on the quasi-periodicities. When the absorbing parameter tends to $0$, the integrand may become irregular as 
poles may approach the unit circle during this process. Based on \cite{Hoang2011,Joly2006,Fliss2009a}, we replace the unit circle by a small modification of it. For the choice of the new curve, we require that the integral does not change for any sufficiently small absorbing parameter, and the integrand does not have any poles on the new curve when the parameter is $0$. We also show that when the wavenumber satisfies certain conditions, such a curve always exists.  Luckily, we have proved that the wavenumbers such that the conditions are not satisfied compose a discrete subset of $\R_+$. Thus for any positive wavenumber except for this discrete set, we can write the LAP solution as a contour integral on a closed curve, where the integrand is a family of quasi-periodic solutions of cell problems. As the quasi-periodic cell problems are classical, the numerical simulation of LAP solutions is also easily carried out. The curve can be easily chosen as a piecewise analytic one, as there are only finite number of eigenvalues on the unit circle. Finally, as the solution depends piecewise analytically on the curve, a high order numerical method is designed based on the contour integral. Moreover, we can also extend this method to more general wavenumbers, with an interpolation technique inspired by the paper \cite{Ehrhardt2009a}.

The numerical method is also extended to  halfguide problems. From \cite{Zhang2019a}, any LAP solution of a halfguide problem is (not uniquely) extended to an LAP solution of a fullguide problem. Thus the numerical method can be designed as two steps. The first step is to find out the source term for the fullguide problems such that its solution approximates that of the halfguide problem, for given boundary data. We compute the boundary values for basis functions in the space where the source term lies in, and then find out the corresponding coefficients by the least square method that approximates the boundary data. As the choice of source terms is not unique and the problem is severely ill-posed, a Tikhonov regularization technique is adopted. With this source term, we go to the second step, i.e., to solve the fullguide problem with the method developed in this paper and approximate the solution of the halfguide problem. Moreover, the methods can also be extended to more general wavenumbers in the same way.

The rest of this paper is organized as follows. We recall the mathematical model of the scattering problem in the second section, and introduce the Floquet-Bloch theory in the third section. Then we consider the quasi-periodic cell problems in the fourth section. In Section 5 and 6, we apply the Floquet-Bloch transform to obtain a simplified integral representation of the LAP solutions. In Section 7, we develop a numerical method to compute the numerical solutions, based on the integral representation. Then we extend the method to halfguide problems in Section 8 and more general wavenumbers in Section 9. In the last section, we present numerical results to show the efficiency of our algorithms.

\section{The mathematical model of direct scattering problems}
\label{sec:math_model}

Let $\Omega=\R\times\Sigma\subset\R^2$ be a closed waveguide, where $\Sigma\subset\R$ is an interval in $\R$ (see Figure \ref{waveguide}). In this paper, we set $\Sigma=(0,1)$ for simplicity. So the boundary of $\Omega$ is
\[\partial\Omega=\{x\in\R^2:\,x_2=0\text{ or }x_2=1\},\]
and it is assumed to be impenetrable. 
Suppose $\Omega$ is filled up with periodic material with a real-valued refractive index $q$ \high{satisfying} the following conditions:
\begin{equation*}
 q(x_1+1,x_2)=q(x_1,x_2),\quad q\geq c_0>0 \quad\forall\,x\in\Omega.
\end{equation*}
\high{Moreover, we require that $q\in L^2_{loc}(\Omega)$.}

\begin{figure}[ht]
\centering
\includegraphics[width=12cm,height=2.5cm]{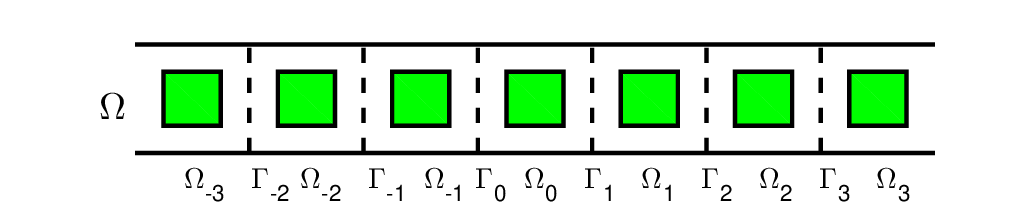}
\caption{Periodic waveguide.}
\label{waveguide}
\end{figure}

The scattering problem is modelled by the following equations:
\begin{eqnarray}
 \Delta u+k^2 q u&=&f\quad\text{ in }\Omega;\label{eq:wg1}\\
 \frac{\partial u}{\partial x_2}&=&0\quad\text{ on }\partial\, \Omega,\label{eq:wg2}
\end{eqnarray}
where  $f$ is a function in $L^2(\Omega)$ with compact support, \high{and $k>0$ is a real and positive wavenumber}.
\begin{remark}
 We can also consider different boundary conditions on $\partial\,\Omega$, e.g., the Dirichlet boundary condition or the Robin boundary condition, in the similar way. In this paper, we only want to take the Neumann boundary condition as an example for our methods.
\end{remark}


To find a  ``physically meaningful'' solution of  \eqref{eq:wg1}-\eqref{eq:wg2}, we introduce the well-known Limiting Absorption Principle (LAP). That is,  given any $\epsilon>0$, consider the following {\em damped Helmholtz equation}:
\begin{eqnarray}\label{eq:wg_ep1}
\Delta u_\epsilon+(k^2+\i\epsilon)q u_\epsilon&=&f\quad\text{ in }\Omega;\\\label{eq:wg_ep2}
\frac{\partial u_\epsilon}{\partial\high{x_2}}&=&0\quad\text{ on }\partial\,\Omega.
\end{eqnarray}
From the Lax-Milgram theorem, the problem is uniquely solvable in $H^1(\Omega)$. Moreover, the solution $u_\epsilon$ decays exponentially as $x_1\rightarrow\pm\infty$ (see \cite{Joly2006}).  The LAP assumes that $u_\epsilon$ converges in $H^1_{loc}(\Omega)$ when $\epsilon\rightarrow 0^+$, and the limit, denoted by $u$, is the  ``physically meaningful'' LAP solution. In the following parts, we introduce a new formulation for LAP solutions. Based on the new formulation, we introduce new numerical methods to compute the LAP solutions efficiently.

For simplicity, we introduce  the following operator:
\begin{equation*}
\A u:=-q^{-1} \Delta u
\high{\text{ in } D(\A,\Omega):=\Big\{u\in H^1(\Omega):\,\Delta u\in L^2(\Omega),\,\left.\frac{\partial u}{\partial x_2}\right|_{\partial\,\Omega}=0\Big\}}.
\end{equation*}
Let the spectrum of $\A$ be denoted by $\sigma(\A)$, then the problem \eqref{eq:wg1}-\eqref{eq:wg2} is  uniquely solvable in $H^1(\Omega)$ if and only if $k^2\notin\sigma(\A)$. To study the spectrum property of $\A$ which plays an important role in the well-posedness of the problem \eqref{eq:wg1}-\eqref{eq:wg2}, we introduce the Floquet-Bloch theory. For details we refer to \cite{Joly2006,Fliss2015} and for more general results we refer to \cite{Kuchm1993}.

\section{Floquet-Bloch theory and quasi-periodic problems}
\label{sec:f-b-theory}

\subsection{Floquet-Bloch theory}
\label{sec:f-b-theory}

In this section, we introduce the Floquet-Bloch theory to study the spectrum $\sigma(A)$. For simplicity, we introduce the following domains (see Figure \ref{waveguide}):
\[ \Omega_j:= (-1/2+j,1/2+j]\times\Sigma,\quad \Gamma_j=\{-1/2+j\}\times \Sigma,\quad j\in\Z.\]
 Then $\Omega=\bigcup_{j\in\Z}\Omega_j$ with the left and right boundaries $\Gamma_j$ and $\Gamma_{j+1}$. Let 
 \[\partial\,\Omega_j:=\partial\,\Omega\cap\overline{\Omega_j}=\Big\{x\in\R^2:\,-1/2+j<x_1\leq 1/2+j,\,x_2=0\text{ or }1\Big\}. \]

We also introduce the space of quasi-periodic functions.  A function $\phi\in H^1_{loc}(\Omega)$ is called $z$-quasi-periodic, if it satisfies
\begin{equation}\label{eq:z_quasi}
\left.\phi\right|_{\Gamma_{j+1}}=z\, \left.\phi\right|_{\Gamma_j},\quad  \left.\frac{\partial \phi}{\partial x_1}\right|_{\Gamma_{j+1}}=z\left.\frac{\partial \phi}{\partial x_1}\right|_{\Gamma_{j}},\quad \forall j\in\Z
\end{equation}
for some fixed complex number $z\in\C$. We define the subspace of $H^1(\Omega_0)$ by:
\begin{equation*}
H_z^1(\Omega_0):=\Big\{\phi\in H^1(\Omega_0):\,\phi\text{ satisfies \eqref{eq:z_quasi} for $j=0$}\Big\}.
\end{equation*}
\high{Then the functions in $H_z^1(\Omega_0)$ can be extended to $z$-quasi-periodic functions.}   
Especially, when $z=1$, all functions in $H^1_1(\Omega_0)$ can be extended into  periodic functions in $H^1_{loc}(\Omega)$. We also denote $H^1_1(\Omega_0)$ by $H^1_\p(\Omega_0)$.

From the Floquet-Bloch theory, the spectrum of $\A$ is closely related to {\em Bloch wave solutions}. \high{A Bloch wave solution is a non-trivial $z$-quasi-periodic solution of \eqref{eq:wg1}-\eqref{eq:wg2} in $H^1_{loc}(\Omega)$ with $f=0$ for some $z\in\C$.}  If a Bloch wave solution exists in $H^1_z(\Omega_0)$,  $z$ is called a {\em Floquet multiplier}.  Define the operator:
\begin{equation}
\A_z u=-q^{-1}\,\Delta u\,\text{ \high{with domain} }D_z(\A,\Omega_0):=D(\A,\Omega_0)\cap H_z^1(\Omega_0),
\end{equation}
where $D(\A,\Omega_0)$ is defined in the same way as $D(\A,\Omega)$, and $\Omega$ is replaced by the periodic cell $\Omega_0$. Moreover, $\A_z$ is self-adjoint with respect to the $L^2$-space equipped with the weighted inner product $(\phi,\psi)_{L^2,q}=\int_{\Omega_0}q\phi\overline{\psi}\d x$. Let $\sigma(\A_z)$ be the spectrum of $\A_z$, then $k^2\in\sigma(A_z)$ if and only if $z$ is a Floquet multiplier.

 Let $\mathbb F(k^2)$ be the collection of all Floquet multipliers with  wavenumber $k$ and  ${\mathbb U\mathbb F}(k^2):=\mathbb{F}(k^2)\cap\UC$ ($\UC$ is the unit circle in $\C$) be the set of all {\em unit Floquet multipliers}. In this paper, when the wavenumber is fixed, we write $\mathbb{F}$ instead of $\mathbb F(k^2)$ for simplicity. We list the properties of the Floquet multipliers from \cite{Zhang2019a}, for more details we refer to \cite{Kuchm1993,Joly2006,Ehrhardt2009,Fliss2015,Kuchm2016}:
\begin{itemize}
\item $\mathbb{UF}$ has at most finite number of elements.
\item $z\in{\mathbb F}$ if and only if $z^{-1}\in{\mathbb F}$, thus $z\in{\mathbb {UF}}$ if and only if $\overline{z}=z^{-1}\in{\mathbb {UF}}$. 
\item  $\mathbb{F}$ is a discrete set and the only finite  accumulation point of $\mathbb{F}$ can be $0$.
\item $\mathbb F(k^2)$ depends continuously on $k^2$.
\end{itemize}

A classical result from the Floquet-Bloch theory also shows that (see \cite{Kuchm1993}):
\begin{equation}\label{eq:A_sig_A}
\sigma(\A)=\bigcup_{|z|=1}\sigma(\A_z).
\end{equation}
Thus, it is particularly important to study the spectrum of $\A_z$ when $|z|=1$. 
For simplicity, let $\alpha=-\i\log(z)$ where $\alpha\in(-\pi,\pi]$. We  replace $\A_z$ by $\A_\alpha$ in the rest of this section, by abuse of notation. Then \eqref{eq:z_quasi} becomes
\begin{equation}
\label{eq:alpha_quasi}
\left.u\right|_{\Gamma_{j+1}}=\exp(\i\alpha) \left.u\right|_{\Gamma_j},\quad  \left.\frac{\partial u}{\partial x_1}\right|_{\Gamma_{j+1}}=\exp(\i\alpha) \left.\frac{\partial u}{\partial x_1}\right|_{\Gamma_{j}},\quad \forall j\in\Z.
\end{equation}
Denote the spectrum of $\A_\alpha$ by $\sigma(\A_\alpha)$. As $\A_\alpha$ is self-adjoint, $\sigma(\A_\alpha)$ is a discrete subset of $(0,\infty)$.  By rearranging the order of the points in $\sigma(\A_\alpha)$ properly, we obtain a family of analytic functions $\{\mu_n(\alpha):\,n\in\N\}$ and $\{\psi_n(\cdot,\alpha):\,n\in\N\}$:
 \begin{equation*}
\A_\alpha\psi_n(\cdot,\alpha)=\mu_n(\alpha)\psi_n(\cdot,\alpha),\quad \sigma(\A_\alpha)=\bigcup_{n\in\N}\{\mu_n(\alpha)\}.
 \end{equation*}
Note that the analytic functions in normed spaces are defined as follows.
\begin{definition}
	Suppose the function $\phi(z,x)$ satisfies $\phi(z,\cdot)\in X$ for any fixed $z$, for some normed space $X$. Then $\phi$ depends analytically on $z$ in an open domain $U\subset\C$ if for any fixed $z_0\in U$, there exist $\phi_\ell\in X$ such that
	\begin{equation*}
		\phi(z,x)=\sum_{\ell=0}^\infty (z-z_0)^\ell\phi_\ell
	\end{equation*}
	converges uniformly in $B(z_0,\delta)$ for a sufficiently small $\delta>0$ with respect to the norm of $X$.
\end{definition}

 Thus $\sigma(\A)=\bigcup_{n\in\N}\bigcup_{\alpha\in(-\pi,\pi]}\big\{\mu_n(\alpha)\big\}$.  Both $\mu_n(\alpha)$ and $\psi_n(\cdot,\alpha)$ are extended into analytic functions in $\alpha$ in a sufficiently small neighbourhood of $(-\pi,\pi)\times\{0\}$. 

For any fixed $n\in\N$, the graph $\{(\alpha,\mu_n(\alpha)):\,\alpha\in(-\pi,\pi]\}$ is called a dispersion curve, and  all  dispersion curves compose a dispersion diagram. Following \cite{Ehrhardt2009a,Ehrhardt2009}, we first show the dispersion diagrams for two different examples:
\begin{enumerate}
\item Example 1. $q=1$ is a constant function, and its dispersion diagram is shown in Figure \ref{fig:dd} (left). The dispersion curve is given analytically:
\begin{equation*}
\mu_{jm}(\alpha)=j^2\pi^2+(\alpha+2\pi m)^2,\quad j\in\N,\,m\in\Z.
\end{equation*}
\item Example 2. $q=9$ in a disk $B\big((0,0.5),0.3\big)$ and $q=1$ outside the disk, and its dispersion diagram is shown in Figure \ref{fig:dd} (right).. 
\end{enumerate}

\begin{figure}[ht]
\centering
\begin{tabular}{c  c}
\includegraphics[width=0.45\textwidth]{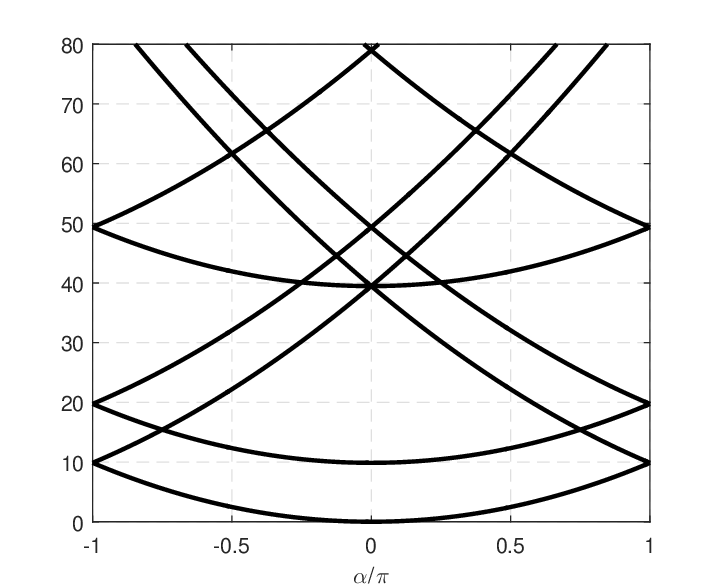} 
& \includegraphics[width=0.45\textwidth]{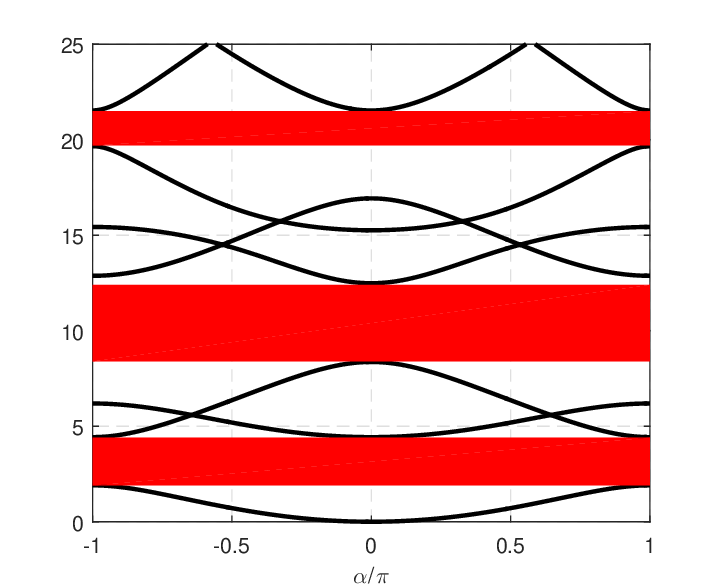}\\[-0cm]
\end{tabular}
\caption{Dispersion diagram. Left: Case I; Right: Case II.}
\label{fig:dd}
\end{figure}

In the right picture of Figure \ref{fig:dd}, there \high{are  ``stop bands''} (in red). When $k^2$ lies in the stop bands, the horizontal line with height $k^2$ has no intersection with any dispersion curves. This implies there is no propagating mode and the scattering problem \eqref{eq:wg1}-\eqref{eq:wg2} has a unique solution in $H^1(\Omega)$.
 The rest of bands are called ``pass bands'', such as the whole domain in the left picure of Figure \ref{fig:dd} and the white region in the right picure of Figure \ref{fig:dd}. 
 When $k^2$ lies in a pass band, from \eqref{eq:A_sig_A}, there is at least one $\alpha\in(-\pi,\pi]$ such that there is a non-trivial $\alpha$-quasi-periodic function $\psi$ \high{satisfying}  $\A_\alpha\psi=k^2\psi$. Thus it is a Bloch wave solution  and is called a {\em propagating Floquet mode}.   The case when $k^2$ lies in a pass band is particularly interesting and \high{challenging. Thus} we discuss more details about this case.

 When $k^2$ lies in a pass band, there is at least one $\alpha\in(-\pi,\pi]$ such that $k^2\in\sigma(A_\alpha)$. Then the set
 \begin{equation*}
  P :=\left\{\alpha\in(-\pi,\pi]:\,\exists\, n\in\N,\,{\rm s.t.,}\, \mu_n(\alpha)=k^2\right\}\neq\emptyset.
 \end{equation*}
Thus $\mathbb{UF}$ can be written as:
\begin{equation*}
 \mathbb{UF}=\left\{\exp(\i\alpha):\,\alpha\in P\right\}.
\end{equation*}
The points in $P $ are divided into the following three classes:
\begin{itemize}
 \item  When $\mu'_n(\alpha)>0$,  $\psi_n(\cdot,\alpha)$  propagates from the left to the right;
 \item when $\mu'_n(\alpha)<0$,  $\psi_n(\cdot,\alpha)$  propagates from the right to the left;
 \item when $\mu'_n(\alpha)=0$, we can not decide the direction that $\psi_n(\cdot,\alpha)$ propagates. 
\end{itemize}
\high{For physical interpretations of the Floquet modes $\psi_n(\cdot,\alpha)$ we refer to Remark 4, \cite{Fliss2015}. }
Based on the above classification, we define the following three sets:
\begin{eqnarray*}
 P_\pm &:=&\left\{\alpha\in(-\pi,\pi]:\,\exists\, n\in\N\text{ s.t., }\mu_n(\alpha)=k^2\text{ and }\pm\mu'_n(\alpha)>0\right\};\\
 P_0 &:=&\left\{\alpha\in(-\pi,\pi]:\,\exists\, n\in\N\text{ s.t., }\mu_n(\alpha)=k^2\text{ and }\mu'_n(\alpha)=0\right\}.
\end{eqnarray*}
Then $P =P_+ \bigcup P_- \bigcup P_0 $. 

\high{\begin{remark}\label{rm1}
It is possible that there are two (or more) different dispersion curves passing through the point $(\alpha,k^2)$. Suppose the elements in $P_+$ have $Q$ different values $\alpha_1,\dots,\alpha_Q$. For any $j=1,2,\dots,Q$, there are $L_j$ ($L_j\geq 1$) different dispersion curves $\mu_{j,\ell}$ ($\ell=1,2,\dots,L_j$) such that
\[
\mu_{j,1}(\alpha_j)=\mu_{j,2}(\alpha_j)=\cdots \mu_{j,L_j}(\alpha_j)=k^2.
\] 
In this case,  $\alpha_j$ is treated as $L_j$ different elements in $P_+$, i.e.,
\[
\alpha_{j,1}=\alpha_{j,2}=\cdots=\alpha_{j,L_j}=\alpha_j.
\]
\end{remark}}

As $\mathbb{UF}$ is symmetric, $P_\pm$ is also symmetric, i.e.,  $\alpha\in P_+ $ if and only if $-\alpha\in P_- $. For details see 
Theorem 4, \cite{Fliss2015}.

As the limiting absorption principle fails when the set $P_0 $ is not empty, we  make the following assumption.

\begin{assumption}\label{asp1}
 Assume that in this paper,  $P_0 =\emptyset$.
\end{assumption}
The assumption is reasonable as the set $\big\{k>0:\,P_0 \neq\emptyset\big\}$ is ``sufficiently small'', i.e., the set is countable with at most one accumulation point at $\infty$ (see Theorem 5, \cite{Fliss2015}).

In our later proof of the new integral representation of LAP solutions, we also have to avoid the cases when $P_+ \cap P_- \neq\emptyset$. Luckily, with the similar method used in the proof of Theorem 5 in \cite{Fliss2015}, we can also prove that this set is discrete in the following lemma. For the proof we refer to Appendix.

\begin{lemma}\label{th:accumulation}
The set $\big\{k\in\R_+:\,P_+ \cap P_- \neq\emptyset\big\}$ is countable, and its only accumulation point is $\infty$.
\end{lemma}


\begin{assumption}\label{asp2}
In Section 1-\ref{sec:special}, we assume that $k$ satisfies $P_+ \cap  P_- =\emptyset$.
\end{assumption}

With Assumption \ref{asp1} and \ref{asp2}, when $\alpha\in(-\pi,\pi]$ is an element in $P$, the propagating mode corresponds to $\alpha$ either travels to the left or to the right. This implies that the propagating modes that travel to the left or right are ``separated''. However,  Assumption \ref{asp2} is not a necessary condition for the LAP. The only reason that we make this assumption is to guarantee our ``simplified representation'' for the LAP solution works. However, although the ``simplified representation'' works for almost all positive wavenumbers, we still discuss the case without Assumption \ref{asp2}  in Section \ref{sec:special}.

We define three subsets of $\mathbb{UF}$ from the definition of $P_\pm $ and $P_0 $  by:
\begin{equation}\label{eq:set_S}
S_\pm^0:=\left\{z=\exp(\i\alpha):\,\alpha\in P_\pm \right\},\quad S_0^0:=\left\{z=\exp(\i\alpha):\,\alpha\in P_0 \right\}.
\end{equation} 
From Remark \ref{rm1}, there may be more than one elements in $P_\pm$ with only one value $\alpha$. In this case, the corresponding elements in $S_\pm^0$ are also different. 
Then $\mathbb{UF}=S_+^0\bigcup S_-^0\bigcup S_0^0$. From the definitions of $P_\pm $, when $z\in S_+^0$, the corresponding Bloch wave solution is propagating to the right; while when $z\in S_-^0$, the corresponding Bloch wave solution is propagating to the left. See Figure \ref{fig:z_alpha} for the unit Floquet multipliers in both $\alpha$- and $z$-space. 

\begin{figure}[ht]
	\centering
	\begin{tabular}{c  c}
		\includegraphics[width=0.5\textwidth,height=0.36\textwidth]{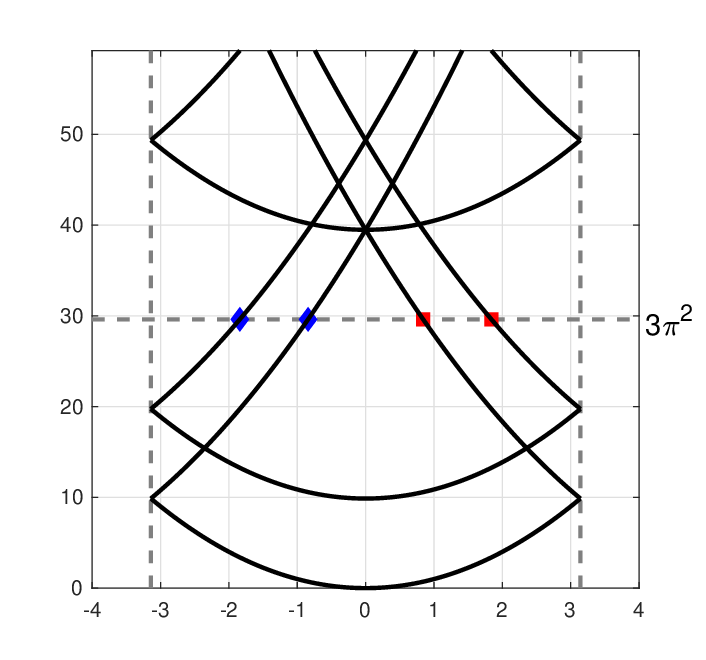} 
		& \includegraphics[width=0.4\textwidth,height=0.36\textwidth]{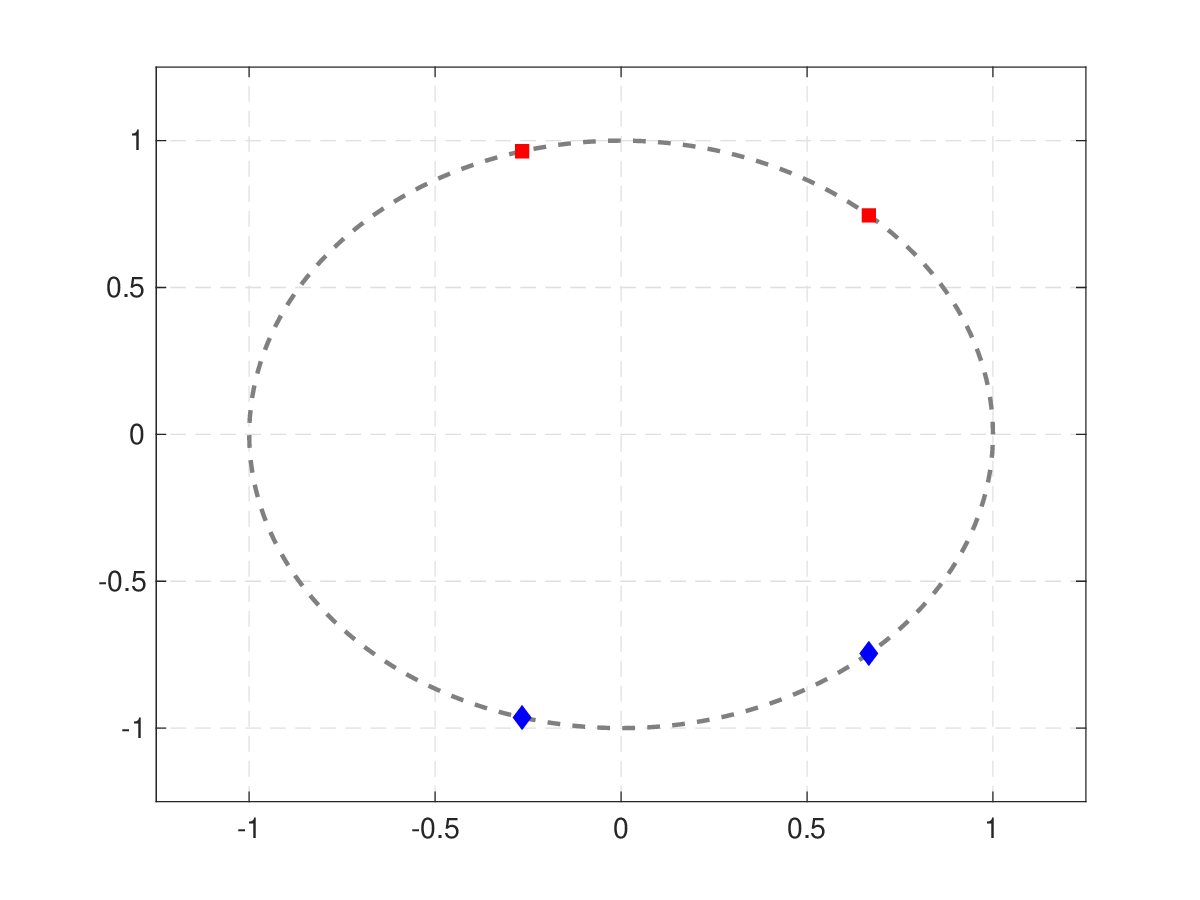}\\[-0cm]
	\end{tabular}
	\caption{Example for $n=1$ and $k^2=3\pi^2$. $\mathbb{UF}$ in $\alpha$-space and $z$-space. Red squares denotes the points in $P_-$ ($\S_-^0$), while blue diamonds denotes the points in $P_+$ ($\S_+^0$).}
	\label{fig:z_alpha}
\end{figure}

We also divide the set $\mathbb{F}\setminus\mathbb{UF}$ into the following two subsets:
\begin{equation*}
RS:=\left\{z\in\mathbb{F}:\,|z|<1\right\};\quad LS:=\left\{z\in\mathbb{F}:\,|z|>1\right\}.
\end{equation*}
The Bloch wave solution corresponds to $z\in RS$ is evanescent, while the one corresponds to $z\in LS$ is anti-evanescent.
Moreover, let
\begin{equation*}
S_+:=S_+^0\bigcup RS,\quad S_-:=S_-^0\bigcup LS.
\end{equation*}

\begin{remark}\label{rm:asp2}
	We conclude the properties of the sets $RS$, $LS$ and $S_\pm^0$ from the properties of $\mathbb{F}$ as follows:
	\begin{itemize}
		\item From the symmetry of $\mathbb{F}$, the sets $S_+^0$ and $S_-^0$, $RS$ and $LS$ are symmetric, i.e.,
		\begin{equation}\label{eq:symm}
		z\in S_+^0\iff z^{-1}=\overline{z}\in S_-^0;\quad z\in RS\iff z^{-1}\in LS.
		\end{equation}
		This also implies that
		\begin{equation*}
		z\in S_+\iff z^{-1}\in S_-. 
		\end{equation*}
		\item When Assumption \ref{asp1} is satisfied, $S_0^0=\emptyset$, thus $S_+^0\bigcup S_-^0=\mathbb{UF}$. 
		\item When Assumption \ref{asp2} is satisfied, $S_+^0\cap S_-^0=\emptyset$.
		\item From \eqref{eq:symm}, if $\pm1\in S_+^0$ (or $\pm1\in S_-^0$), then $\pm 1\in S_-^0\cap S_+^0$. If Assumption \ref{asp2} is satisfied, $S_+^0\cap S_-^0=\emptyset$, then $\pm1\notin\mathbb{UF}\setminus S_0^0$. If Assumption \ref{asp1} is also satisfied, $S_0^0=\emptyset$ implies that $\pm 1\notin\mathbb{F}$.
	\end{itemize}

\end{remark}


\begin{lemma}\label{th:s_pm_dist}
	Let $k>0$. 
	There is a $\tau>0$ such that $RS\in B(0,\exp(-\tau))$ and $LS\in\C\setminus\overline{B(0,\exp(\tau))}$.
\end{lemma}

\subsection{Quasi-periodic problems}
\label{sec:qp-prb}

From the last section, quasi-periodic problems are very important in the investigation of scattering problems in periodic domains. 
In this section, we consider the $z$-dependent cell problem:
\begin{eqnarray}\label{eq:quasi1}
 \Delta u_z+k^2 qu_z&=&f_z\quad\text{ in }\Omega_0;\\\label{eq:quasi2}
 \frac{\partial u_z}{\partial \high{x_2}}&=&0\quad\text{ on }\partial\,{\Omega_0};
\end{eqnarray}
where $z\neq 0$ is a complex number and $f_z\in L^2(\Omega_0)$ depends analytically on $z$. The analytical dependence on $z$ is defined as follows.

We are interested in the well-posedness of the problem \eqref{eq:quasi1}-\eqref{eq:quasi2}, and also the dependence of the solution $u_z$ on the quasi-periodicity $z$. To this end, it is more convenient to study the problems in a fixed domain. 
Let $v_z=z^{-x_1}u_z(x)$, then $v_z\in H_\p^1(\Omega_0)$.
Note that as $z^{-x_1}$ is a multi-valued function, we require that $z$ lies in the branch cutting off along the negative real axis  (denoted by $\C_\times:=\{z\in\C\setminus\{0\}:\,-\pi<\arg(z)\leq \pi\}$, where $\arg(z)$ is the argument of the complex number $z$).
From direct calculation, $v_z$ is the solution of the following problem:
\begin{eqnarray*}
 \Delta v_z+2\log(z)\frac{\partial v_z}{\partial x_1}+(k^2q+\log^2(z))v_z&=&z^{-x_1} f_z\quad\text{ in }\Omega_0;\\
 \frac{\partial v_z}{\partial \nu}&=&0\quad\text{ on }\partial\,{\Omega_0}.
\end{eqnarray*}
 Then the variational formulation of the periodized problem, is to find  a solution $v_z\in H_\p^1(\Omega_0)$ such that it satisfies
\begin{equation}\label{eq:quasi_periodic_var}
\begin{aligned}
\int_{\Omega_0}\left[\grad v_z\cdot\grad\overline{\phi}+\log(z)\left(v_z\frac{\partial\overline{\phi}}{\partial x_1}-\frac{\partial v_z}{\partial x_1}\overline{\phi}\right)-(k^2 q+\log^2(z))v_z\overline{\phi}\right]\d x\\
=-\int_{\Omega_0}\left(z^{-x_1} f_z\right)\overline{\phi}\d x
\end{aligned}
\end{equation}
for any $\phi\in H_\p^1(\Omega_0)$.

The left hand side is a sesquilinear form defined in $H_\p^1(\Omega_0)\times H_\p^1(\Omega_0)$. From Riesz's lemma, there is a bounded linear operator $\K_z\in\mathcal{L}(H_\p^1(\Omega_0))$ and a function $\widetilde{f}_{z}\in H_z^1(\Omega_0)$ such that
\begin{eqnarray}
&&\high{\begin{aligned}
&\left<\K_z v_z,\phi\right>= \\%
	&\quad-\int_{\Omega_0}\left[\log(z)\left(v_z\frac{\partial\overline{\phi}}{\partial x_1}-\frac{\partial v_z}{\partial x_1}\overline{\phi}\right)-(k^2 q+1+\log^2(z))v_z\overline{\phi}\right]\d x;
\end{aligned}}
\\\label{eq:tilde_f}
&&\left<z^{-x_1}\widetilde{f}_z,\phi\right>=-\int_{\Omega_0}\left(z^{-x_1}f_z\right)\overline{\phi}\d x;
\end{eqnarray}
for all $\phi\in H_\p^1(\Omega_0)$, 
where $\left<\cdot,\,\cdot\right>$ is the inner product defined in $H_\p^1(\Omega_0)$.
Thus the variational form  \eqref{eq:quasi_periodic_var} is equivalent to 
\begin{equation*}
(I-\K_z)v_z=z^{-x_1}\widetilde{f}_z.
\end{equation*}
This implies that when $I-\K_z$ is invertible, then
\[
 v_z=\left(I-\K_z\right)^{-1}z^{-x_1}\widetilde{f}_z.
\]

Now we focus on the inverse operator of $I-\K_z$. As $\K_z$ is compact and depends analytically on $z\in\C_\times$, $I-\K_z$ is an analytic family of Fredholm operators. We recall the following result from the {\em Analytic Fredholm Theory}.

\begin{theorem}[Theorem VI.14, \cite{Reed1980}]
 Let $D$ be an open connected domain in $\C$, $X$ be a Hilbert space, and $\T:\,D\rightarrow \mathcal{L}(X)$ be an operator valued analytic function such that $\T(z)$ is compact for each $z\in D$. Then either
  \begin{itemize}
  \item $(I-\T(z))^{-1}$ does not exist for any $z\in D$, or
  \item Let the set $S=\left\{z\in D:\,I-T(z)\text{ is not one-to-one}\right\}$. Then $S$ is a discrete subset of $D$. In this case, $(I-\T(z))^{-1}$ is meromorphic in $D$ and analytic in $D\setminus S$. The residues at the poles are finite rank operators.
 \end{itemize}
 \label{th:ana_fred_thy}
\end{theorem}

From Section in \cite{Zhang2019a}, the set of poles of $\left(I-\K_z\right)^{-1}$ is exactly $\mathbb{F}$. Thus 
$v_z$, or equivalently $u_z=z^{x_1}v(x)$, depends analytically on $z\in\C\setminus\left(\mathbb{F}\bigcup\{0\}\right)$ and meromorphically on $z\in\C\setminus\{0\}$.

\section{The Floquet-Bloch transform and its application}
\label{sec:f-b-trans-app}

\subsection{The Floquet-Bloch transform}
\label{sec:f-b-trans}

The Floquet-Bloch transform is a very important tool in the analysis of scattering problems in PDEs in periodic structures, see \cite{Kirsc2017,Kirsc2017a,Fliss2015}. 
 



For a function $\phi\in C_0^\infty(\Omega)$, define the Floquet-Bloch transform of $\phi$ by
\begin{equation}\label{eq:def_F}
 (\F \phi)(z,x):=\sum_{n=-\infty}^\infty \phi(x_1+n,x_2)z^{-n},\quad x\in\Omega_0,\,z\in\C.
\end{equation}
The transform is well-defined for any smooth function with compact support, and it can be extended to more general cases. Define the {\em Region of Convergence (ROC)} as the domain in $\C$ such that the series \eqref{eq:def_F} converges. Note that the DOC may be empty for given function $\phi$. 
When $\phi$ {\em decays exponentially at the rate $\gamma$}, i.e., there is a $\gamma>0$ and $C>0$ such that $\phi$ satisfies
\begin{equation}\label{eq:exp_decay}
|\phi(x_1,x_2)|\leq C\exp(-\gamma|x_1|),\quad\forall\,x\in\Omega,
\end{equation} 
the Floquet-Bloch transform of $\phi$ is still well-defined, and the ROC is the annulus
\begin{equation*}
T_\gamma=\big\{z\in\C:\,\exp(-\gamma)<|z|<\exp(\gamma)\big\}.
\end{equation*}
Moreover, the function $(\F \phi)(z,\cdot)$ depends analytically on $z\in T_\gamma$. It is also easy to check that the transformed function $(\F\phi)(z,\cdot)$ is quasi-periodic (i.e., it satisfies \eqref{eq:z_quasi}). We conclude some mapping properties of the operator $\F$ in the following proposition.

\begin{proposition}\label{th:prop_F}
The operator $\F $ has the following properties when $z$ lies on the unit circle $\UC$ (see \cite{Lechl2016,Kuchm2016}):
\begin{itemize}
\item $\F$ is an isomorphism between $H^s(\Omega)$ and $L^2(\UC;H^s_z(\Omega_0))$ ( $s\in\R$), where 
\begin{equation*}
L^2(\UC;H^s_z(\Omega_0)):=\left\{\phi\in\mathcal{D}'(\UC\times\Omega_0):\,\left[\int_{\UC}\left\|\phi(z,\cdot)\right\|^2_{H^s_z(\Omega_0)}\d z\right]^{1/2}<\infty\right\}.
\end{equation*}
\item $\F \phi$ depends analytically on $z\in T_\gamma$, if and only if $\phi$ decays exponentially  with the rate $\gamma$.




\item Given $\psi(z,x):=(\F \phi)(z,x)$ for some $\phi\in H^s(\Omega)$ and satisfies \eqref{eq:exp_decay} for some $\gamma>0$,  the inverse operator $\F$ is given by:
 \begin{equation}\label{eq:def_invF}
  (\F^{-1}\psi)(x_1+n,x_2)=\frac{1}{2\pi\i}\oint_{\UC}\psi(z,x)z^{n-1}\d z
 \end{equation}
\end{itemize}
\end{proposition}


\subsection{Application of the Floquet-Bloch transform}
\label{sec:app-f-b}

In this section, we apply the Floquet-Bloch transform $\F$ to the scattering problem \eqref{eq:wg1}-\eqref{eq:wg2}, when $k^2\notin\sigma(A)$. We are particularly interested in the case that:
\begin{itemize}
 \item for $k>0$, $k^2\notin\sigma(A)$, i.e., $k^2$ lies in a stop band; or
 \item $k^2$ is no longer real, i.e.,  $k^2=k^2_0+\i\epsilon$ for some fixed $k_0>0$ and $\epsilon>0$.
\end{itemize}
When either of the two conditions is satisfied, the problem \eqref{eq:wg1}-\eqref{eq:wg2} is uniquely solvable in $H^1(\Omega)$. Moreover, $u$ decays exponentially at the infinity, i.e., $u$ satisfies \eqref{eq:exp_decay} for some $C>0$ and $\gamma>0$ (see \cite{Ehrha2010}). 

\begin{remark}
 From now on, we assume that ${\rm supp}(f)\subset\Omega_0$. The results are easily extended to cases when ${\rm supp}(f)$ lies in  larger bounded domains.
\end{remark}

We define the Floquet-Bloch transform $w(z,x):=(\F u)(z,x)$, then the transformed field $w(z,\cdot)$ is well-defined  and depends analytically on $z\in T_\gamma$ in $H^1(\Omega_0)$. It is also easy to check that for any  $z\in T_\gamma$, $w(z,\cdot)\in H_z^1(\Omega_0)$ satisfies \eqref{eq:quasi1}-\eqref{eq:quasi2}. Note that the source term in \eqref{eq:quasi1} is $f$, as $(\F f )(z,x)=f(x)$ for any $z\in T_\gamma$.

From the inverse Floquet-Bloch transform and Cauchy integral theorem, the solution of the original problem can be represented as:
\begin{equation*}
\begin{aligned}
 u(x_1+n,x_2)= (\F^{-1} w)(x_1+n,x_2)&=\frac{1}{2\pi \i}\oint_{\UC} w(z,x)z^{n-1}\d z\\
 &=\frac{1}{2\pi\i}\oint_{\mathcal{C}} w(z,x)z^{n-1}\d z,
\end{aligned}
\end{equation*}
where $\mathcal{C}$ is a rectifiable curve in $T_\gamma$ encircling $0$.

From the exponential decay of $u$, $(\F u)(z,\cdot)$ exists and depends analytically on $z\in T_\gamma$. 
On the other hand, from Section \ref{sec:qp-prb},  when $z\in\C\setminus\mathbb{F}$, the problem \eqref{eq:quasi1}-\eqref{eq:quasi2} is uniquely solvable in $H_z^1(\Omega_0)$, and $w(z,\cdot)$ depends analytically on $z\in\C\setminus(\mathbb{F}\bigcup\{0\})$ and meromorphically on $z\in\C\setminus\{0\}$. Thus $(\F u)(z,\cdot)$ is extended meromorphically in $\C\setminus\{0\}$. 
From the analytic continuation, we obtain the following result.

\begin{theorem}
 When $k^2\notin\sigma(\A)$, the Floquet-Bloch transformed field $(\F u)(z,x)$ is extended to an analytic function in $\C\setminus(\mathbb{F}\bigcup\{0\})$ and a meromorphic function in $\C\setminus\{0\}$ by the solution $w(z,\cdot)$ of \eqref{eq:quasi1}-\eqref{eq:quasi2}. 
\end{theorem}

 The integral representation of $u$ is obtained from Cauchy integral theorem.

\begin{theorem}
\label{th:inv_w_u}
Suppose $k^2\notin\sigma(\A)$.  $w(z,\cdot)$ is the solution of \eqref{eq:quasi1}-\eqref{eq:quasi2} for $z\in\C\setminus\mathbb{F}$. Then the solution of \eqref{eq:wg1}-\eqref{eq:wg2} is written as
\begin{equation}\label{eq:u_inv}
u(x_1+n,x_2)= (\F^{-1} w)(x_1+n,x_2)=\frac{1}{2\pi i}\oint_{\mathcal{C}} w(z,x)z^{n-1}\d z,
\end{equation}
where $\mathcal{C}\subset\C$ is a counter-clockwise closed rectifiable  path encircling all the points in $RS(=S_+)$ and does not encircling any point in $LS(=S_-)$.
\end{theorem}



\section{The Limiting absorption principle (LAP)}
\label{sec:lap}

In this section, we consider the case when $k^2\in\sigma(\A)$ with the help of the limiting absorption principle. First we consider the damped Helmholtz equation \eqref{eq:wg_ep1}-\eqref{eq:wg_ep2}.
The corresponding $z$-quasi-periodic problem is formulated as:
\begin{eqnarray}\label{eq:wg_z1}
\Delta w_\epsilon(z,\cdot)+(k^2+\i\epsilon)qw_\epsilon(z,\cdot)&=&f\quad\text{ in }\Omega_0;\\
\label{eq:wg_z2}
\frac{\partial w_\epsilon(z,\cdot)}{\partial\high{x_2}}&=&0\quad\text{ on }\partial\,{\Omega_0}.
\end{eqnarray}
Similar to the definition of $\K_z$, we denote the operator with $k^2+\i\epsilon$ by $\K_z^\epsilon$. From \high{Theorem 4 in \cite{Stein1968}}, the poles of the operator $(I-\K^\high{\epsilon}_z)^{-1}$ depends continuously on $\epsilon$. First, we study the asymptotic behaviour of distributions of the poles when $\epsilon>0$ is sufficiently small.

\subsection{Distribution of poles of the damped Helmholtz equations}
\label{sec:poles}

From the Floquet-Bloch theory (see \cite{Kuchm1993,Fliss2015}), $k^2\in\sigma(\A)$ implies that $\mathbb{UF}\neq\emptyset$. For any $z\in\mathbb{UF}$, $k^2$ is an eigenvalue of $\A_z$. As both $S_+^0$ and $S_-^0$ are finite sets,  and $S_+^0$ and $S_-^0$, $RS$ and $LS$ are symmetric in the sense of \eqref{eq:symm}, they are written as
\begin{eqnarray}\label{eq:set_R}
&S_+^0=\{z_{j,\ell}^+:\,j=1,\dots,Q;\,\ell=1,\dots,L_j\};\quad RS=\{z_{Q+1}^+,z_{Q+2}^+,\dots\};\\
\label{eq:set_L}
&S_-^0=\{z_{j,\ell}^-:\,j=1,\dots,Q;\,\ell=1,\dots,L_j\};\quad LS=\{z_{Q+1}^-,z_{Q+2}^-,\dots\};
\end{eqnarray}
where $z_j^+$ ($j=1,\dots,Q$) are $Q$ different values on the unit circle, and so are $z_j^-$ ($j=1,\dots,Q$). Moreover, $z_{j,1}^\pm=\dots=z_{j,L_j}^\pm=z_j^\pm$ for $j=1,\dots,Q$, and $z_j^+=\left(z_j^-\right)^{-1}$ for any integer $j\in\N$. From Assumption  \ref{asp1} and \ref{asp2}, $\mathbb{UF}=S_+^0\bigcup S_-^0$, $\mathbb{F}=S_+^0\bigcup S_-^0\bigcup RS\bigcup LS$; moreover, $S_+^0\cap S_-^0=\emptyset$. 


From the continuous dependence of poles,  for any $z_{j,\ell}^\pm$ or $z_j^\pm\in\mathbb{F}$ with $j\in\N$, there is a continuous function $Z_{j,\ell}^\pm$ or $Z_j^\pm(\epsilon)$, such that $\{Z_{j,\ell}^+(\epsilon),Z_{j,\ell}^-(\epsilon):\,j=1,\dots,Q;\,\ell=1,\dots,L_j\}\bigcup\{Z_j^+(\epsilon),Z_j^-(\epsilon):\,j\geq Q+1\}$ are exactly the set of all poles with respect to $k^2+\i\epsilon$. Moreover,  $\lim_{\epsilon\rightarrow 0}Z^\pm_{j,\ell}(\epsilon)=Z^\pm_{j,\ell}(0)=z_j^\pm$ for $j=1,\dots,Q$ and $\ell=1,\dots,L_j$ and $\lim_{\epsilon\rightarrow 0}Z^\pm_{j}(\epsilon)=Z^\pm_{j}(0)=z_j^\pm$ for $j\geq Q+1$. Analogous to the case $\epsilon=0$, we define the following sets depending on $\epsilon$:
\begin{eqnarray*}
&&S_+^0(\epsilon)=\{Z_{j,\ell}^+(\epsilon):\,j=1,\dots,Q;\,\ell=1,\dots,L_j\};\\ 
&&S_-^0(\epsilon)=\{Z_{j,\ell}^-(\epsilon):\,j=1,\dots,Q;\,\ell=1,\dots,L_j\};\\ &&LS(\epsilon)=\{Z_{Q+1}^-(\epsilon),Z_{Q+2}^-(\epsilon),\dots\};\\
&&RS(\epsilon)=\{Z_{Q+1}^+(\epsilon),Z_{Q+2}^+(\epsilon),\dots\}.
\end{eqnarray*} 
From the continuous dependence of $\epsilon>0$, we have the following properties of $Z_j^\pm(\epsilon)$ when $j=1,2,\dots,Q$. For the proof we refer to Appendix in \cite{Joly2006}. 

\begin{lemma}\label{th:curve_z}
For any $j=1,2,\dots,Q$, when $\epsilon>0$ is sufficiently small, the functions satisfy $|Z_{j,\ell}^+(\epsilon)|<1$ and $|Z_{j,\ell}^-(\epsilon)|>1$. 
\end{lemma}

%

Thus we conclude the behaviour of  $Z_j^\pm(\epsilon)$ for sufficiently small $\epsilon$:
\begin{itemize}
\item for any $z_j^+\in S^0_+$, the points $Z_{j,\ell}^+(\epsilon)\in S_+^0(\epsilon)$ converges to $z_j^+$ from the interior of the unit circle;
\item for any $z_j^-\in S^0_-$, the points $Z_{j,\ell}^-(\epsilon)\in S_-^0(\epsilon)$ converges to $z_j^-$ from the exterior of the unit circle.
\end{itemize}
To make it clear, we present a visualization of  examples of the curves in Figure \ref{fig:curve}. The red rectangles are points in $S_-^0$ and the blue diamonds are points in $S_+^0$. The asymptotic behavior of $Z_{j,\ell}^\pm(\epsilon)$ as $\epsilon\rightarrow 0$ can be seen from the picture.

For the points in $RS(\epsilon)$ or $LS(\epsilon)$, we estimate their distributions for sufficiently small $\epsilon>0$ in the following lemma (see Lemma 17, \cite{Zhang2019a}).

\begin{lemma}\label{th:curve_z_inside}
Suppose for some $\tau>0$, $RS\subset B(0,\exp(-\tau))$  and $LS\subset\C\setminus\overline{B(0,\exp(\tau))}$. For any $\tau_1\in (0,\tau)$, there exists $\epsilon_0>0$ such that
 $Z_j^+(\epsilon)\in B(0,\exp(-\tau_1))$ and $Z_j^-(\epsilon)\in \C\setminus\overline{B(0,\exp(\tau_1))}$ for any $j\geq Q+1$ and $\epsilon\in(0,\epsilon_0)$.
\end{lemma}

From Lemma \ref{th:curve_z} and \ref{th:curve_z_inside}, when $\epsilon>0$ is sufficiently small, the sets $S_\pm^0(\epsilon)$ and $RS(\epsilon)$, $LS(\epsilon)$ have  following properties:
\begin{equation}\label{eq:dist_pole}
 S_+^0(\epsilon)\bigcup RS(\epsilon)\subset B(0,1);\quad S_-^0(\epsilon)\bigcup LS(\epsilon)\subset \C\setminus\overline{B(0,1)}.
\end{equation}

\subsection{Integral representation of the LAP solution}
\label{sec:inte_lap}

Now we are prepared to consider the LAP solution of \eqref{eq:wg1}-\eqref{eq:wg2} when $k^2\in\sigma(\A)$.  From Theorem \ref{th:inv_w_u}, the solution $u_\epsilon$ ($\forall\,\epsilon>0$) of the damped problem \eqref{eq:wg_ep1}-\eqref{eq:wg_ep2} is represented by the curve integral:
\begin{equation*}
u_\epsilon(x_1+n,x_2)=\frac{1}{2\pi\i}\oint_{\UC}w_\epsilon(z,x)z^{n-1}\d z.
\end{equation*}
But when $\epsilon\rightarrow 0^+$, as the poles in $S_\pm^0(\epsilon)$ approach $\UC$, the integral becomes irregular. From Theorem \ref{th:inv_w_u}, we are aimed \high{at finding} out a proper curve $\mathcal{C}$ to replace the unit circle such that the function $w_\epsilon(z,x)$ is well-posed and converges uniformly with respect to $z$ when $\epsilon\rightarrow 0$.
From the properties of the sets $S_\pm^0(\epsilon)$, $RS(\epsilon)$ and $LS(\epsilon)$, we define a \high{convenient} curve $\mathcal{C}_0$ as follows. 

\begin{definition}\label{def:C_0}
 Let the piecewise analytic curve $\mathcal{C}_0$ be defined as the boundary of the following domain :
\begin{equation*}
B_\delta:=B(0,1)\bigcup\left[\bigcup_{n=1}^Q B(z_j^+,\delta)\right]\setminus\left[\bigcup_{n=1}^Q \overline{B(z_j^-,\delta)}\right],
\end{equation*}
where  $\delta>0$ is sufficiently small such that the following conditions are satisfied:
\begin{itemize}
\item For any $j\neq \ell$,  $B(z_j^\pm,\delta)\cap B(z_\ell^\pm,\delta)=\emptyset$.
\item $\left[\bigcup_{n=1}^Q B(z_j^+,\delta)\right]\bigcup\left[\bigcup_{n=1}^Q B(z_j^-,\delta)\right]\subset T_\tau$, i.e., the balls do not contain any point in $RS\bigcup LS$.
\end{itemize}
\end{definition}


From Assumption \ref{asp2} and Lemma \ref{th:s_pm_dist}, 
we can always find a proper parameter $\delta$ such that both of the conditions are satisfied.
We refer to Figure \ref{fig:curve} for a visualization of  $\mathcal{C}_0$ for the example when $n=1$ and $k^2=3\pi^2$.

Thus from Lemma \ref{th:curve_z} and \ref{th:curve_z_inside}, there is a constant $C=C(\delta)>0$ such that 
\begin{equation*}
d\left(\mathcal{C}_0,\,RS(\epsilon)\bigcup LS(\epsilon)\bigcup S_+^0(\epsilon)\bigcup S_-^0(\epsilon)\right)>C(\delta)
\end{equation*}
holds uniformly for any fixed sufficiently small $\epsilon>0$, 
where $d(X,Y)$ is the Hausdorff distance between two subsets $X,\,Y\subset\C$. From the choice of the curve $\mathcal{C}_0$, the interior of the symmetric difference of $B_\delta$ and $B(0,1)$ is 
\begin{equation*}
 \Big[\bigcup_{n=1}^Q\left(B(z_n^+,\delta)\setminus\overline{B(0,1)}\right)\Big]\bigcup\Big[\bigcup_{n=1}^Q\left(B(z_n^-,\delta)\cap{B(0,1)}\right)\Big].
\end{equation*}
As for sufficiently small $\epsilon>0$, $(I-\K_z^\epsilon)^{-1}$ depends analytically on $z$ in the above domain, from Cauchy integral theorem, $u_\epsilon$ has the equivalent formulation
\begin{equation*}
u_\epsilon(x_1+n,x_2)=\frac{1}{2\pi\i}\oint_{\mathcal{C}_0}w_\epsilon(z,x)z^{n-1}\d z.
\end{equation*}

\begin{figure}[ht]
\centering
\begin{tabular}{c  c}
\includegraphics[width=0.45\textwidth]{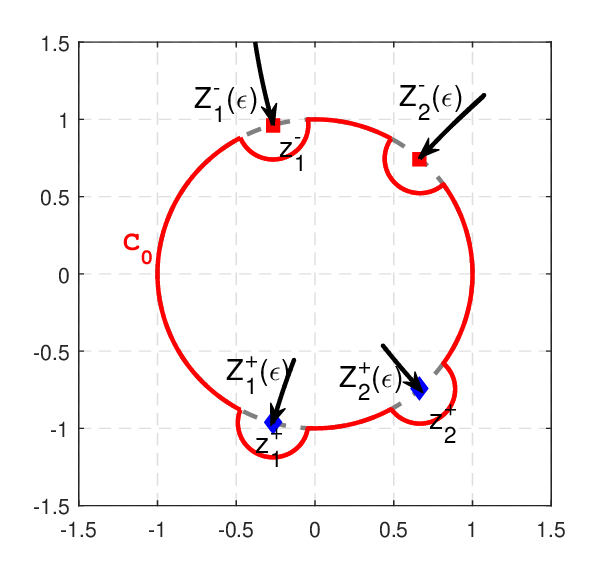} 
& \includegraphics[width=0.45\textwidth]{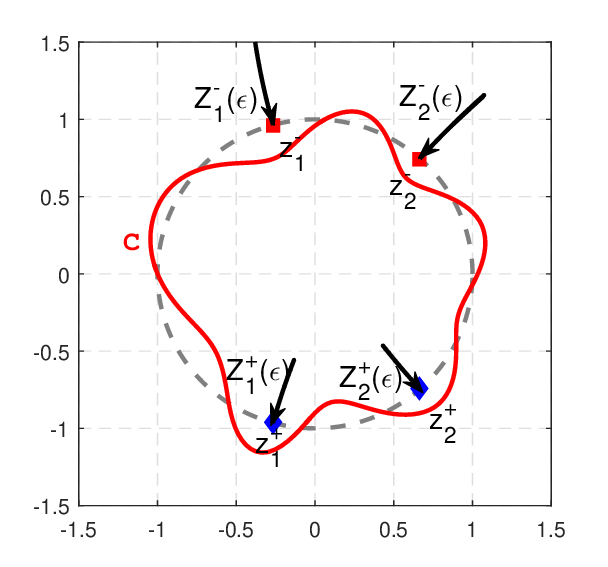}\\[-0cm]
\end{tabular}
\caption{Left: the curve $\mathcal{C}_0$; Right: a choice of $\mathcal{C}$. Red solid curves: $Z_j^\pm(\epsilon)$; grey dashed curves: the unit circle.  The red rectangles are points in $S^0_-$ and the blue diamonds are points in $S^0_+$. The arrows show the direction of the poles when $\epsilon\rightarrow 0$.}
\label{fig:curve}
\end{figure}

\begin{theorem}\label{th:lim_w_eps}
When $\epsilon>0$ is sufficiently small, the function $w_\epsilon(z,x)$ \high{is analytic in an open neighbourhood of $\mathcal{C}_0$} and is uniformly bounded with respect to $\epsilon$. 
Moreover, \begin{equation*}
\lim_{\epsilon\rightarrow 0^+}w_\epsilon(z,\cdot)=w(z,\cdot)\quad\text{ in }H^1(\Omega_0)
\end{equation*}
uniformly for $z$ in the neighourhood of $\mathcal{C}_0$ for any $\gamma_1\in(0,\gamma)$.
\end{theorem}

\begin{proof}
As $d\left(\mathcal{C}_0,\,RS(\epsilon)\bigcup LS(\epsilon)\bigcup S_+^0(\epsilon)\bigcup S_-^0(\epsilon)\right)>C(\delta)$, there is a neighbourhood of $\mathcal{C}_0$ such that $I-\K_z^\epsilon$ is invertible for any $z$ in the neighourhood and $\epsilon$ sufficiently small. Thus $w_\epsilon(z,\cdot)$ is uniformly bounded in $H^1(\Omega_0)$. 
The limit of $w_\epsilon(z,\cdot)$ is proved by the continuity with respect to $\epsilon$. The proof is finished.
\end{proof}

With the above result, we obtain the following integral representation of the solution $u$ \high{of \eqref{eq:wg1}-\eqref{eq:wg2}}.

\begin{theorem}\label{th:LAP}
Suppose Assumption \ref{asp1} and  \ref{asp2} are satisfied, and $\mathcal{C}_0$ is defined by Definition \ref{def:C_0}. Given any compactly supported function $f\in L^2(\high{\Omega})$, and $u_\epsilon\in H^1(\Omega)$ is the unique solution of \eqref{eq:wg_ep1}-\eqref{eq:wg_ep2}. Then
\begin{equation*}
\lim_{\epsilon\rightarrow 0^+}u_\epsilon=u\quad\text{ in }H^1_{loc}(\Omega),
\end{equation*}
where the LAP solution $u$ has the integral representation
\begin{equation}\label{eq:u_LAP}
 u(x_1+n,x_2)=\frac{1}{2\pi\i}\oint_{\mathcal{C}_0}w(z,x)z^{n-1}\d z.
\end{equation}
Moreover, for any $n\in\Z$, there is a constant $C=C(n)>0$ such that
\begin{equation}\label{eq:bdd}
 \|u\|_{H^1(\Omega_n)}\leq C \|f\|_{L^2(\Omega_0)}.
\end{equation}
Moreover, from regularity of elliptic equations, the LAP solution $u\in H^2_{loc}(\Omega)$.
\end{theorem}

\begin{proof}
From Lemma \ref{th:lim_w_eps}, $w_\epsilon(z,\cdot)$ is uniformly bounded with respect to $\epsilon$ and $z$, and converges to $w(z,\cdot)$ in $H^1_{loc}(\Omega)$. Then from the Lebesgue's Dominated Convergence theorem, exchange the limit and integral, \eqref{eq:u_LAP} is proved. From the uniform boundedness of the operator $(I-\B_z)^{-1}$, the function $w(z,\cdot)$ is also uniformly bounded in $H^1(\Omega_0)$ with respect to $z$. Thus \eqref{eq:bdd} is proved for any fixed $n\in\Z$. The proof is finished. 
\end{proof}

We can also replace $\mathcal{C}_0$ by any closed curve that lies in the neighbourhood of $\mathcal{C}_0$ and enclose zero and all poles in $S_+$, but no poles in $S_-$ (see Figure \ref{fig:curve} (right)). The left curve is $\mathcal{C}_0$, and the right is a choice of $\mathcal{C}$. Thus 
\begin{equation}\label{eq:inv_u_eps}
u(x_1+n,x_2)=\frac{1}{2\pi\i}\oint_{\mathcal{C}}w(z,x)z^{n-1}\d z.
\end{equation}

Now we have already  formulate the LAP solution directly from  cell problems \eqref{eq:wg_z1}-\eqref{eq:wg_z2} without the LAP process. This also provides a nice and clear formulation for the numerical scheme. However, we still need to know the dispersion diagram to find the correct curve $\mathcal{C}_0$ (or $\mathcal{C}$).

\subsection{Propagating Bloch wave solutions}
\label{sec:prop_bloch}

From the integral representation of the LAP solution in \eqref{eq:inv_u_eps}, we also have explicit formulations for propagating Bloch wave solutions. Recall that  
\[S_\pm^0=\{z_{j,\ell}^\pm:\,j=1,\dots,Q;\,\ell=1,\dots,L_j\},
\]
and \high{correspondingly}  
\[
P_\pm=\{\alpha_{j,\ell}^\pm:\,j=1,\dots,Q;\,\ell=1,\dots,L_j\}.
\]
Let $\mu_{j,\ell}^\pm(\alpha)$ ($j=1,\dots,Q$ and $\ell=1,\dots,L_j$) be the eigenvalues that depend on $\alpha$, and $\mu_{j,\ell}^+(\alpha_{j,\ell}^+)=\mu_{j,\ell}^-(\alpha_{j,\ell}^-)=k^2$. As Assumption \ref{asp1} holds, $\left(\mu_{j,\ell}^+\right)'(\alpha_{j,\ell}^+)>0$ and $\left(\mu_{j,\ell}^-\right)'(\alpha_{j,\ell}^-)<0$. Let $\mu_\ell(\alpha)$ be other eigenvalues where $\ell=Q+1,\dots$. Then $\mu_\ell(\alpha)\neq k^2$ for all $\alpha\in(-\pi,\pi]$. Let $\psi_{j,\ell}^\pm(\alpha,\cdot)$ ($j=1,\dots,Q$ and $\ell=1,\dots,L_j$) and $\psi_j(\alpha,\cdot)$ ($j=Q+1,\dots$) be  corresponding eigenfunctions. Replace $z$ by $\exp(\i\alpha)$, 
from the definition of resolvent, for $x\in\Omega_0$:
\[
\begin{aligned}
\high{\widetilde{w}}(\alpha,x)=&\sum_{j=1}^Q\sum_{\ell=1}^{L_j}\frac{\left<q^{-1}f,\phi_{j,\ell}^+(\alpha,\cdot)\right>_{L^2,q}}{\mu_{j,\ell}^+(\alpha)-k^2}\phi_{j,\ell}^+(\alpha,\cdot)\\
&+\sum_{j=1}^Q\sum_{\ell=1}^{L_j}\frac{\left<q^{-1}f,\phi_{j,\ell}^-(\alpha,\cdot)\right>_{L^2,q}}{\mu_{j,\ell}^-(\alpha)-k^2}\phi_{j,\ell}^-(\alpha,\cdot)\\&+\sum_{j=Q+1}^\infty\frac{\left<q^{-1}f,\phi_j(\alpha,\cdot)\right>_{L^2,q}}{\mu_j(\alpha)-k^2}\phi_j(\alpha,\cdot)\\
&=\sum_{j=1}^Q\sum_{\ell=1}^{L_j}\frac{\left<f,\phi_{j,\ell}^+(\alpha,\cdot)\right>}{\mu_{j,\ell}^+(\alpha)-k^2}\phi_{j,\ell}^+(\alpha,\cdot)+\sum_{j=1}^Q\sum_{\ell=1}^{L_j}\frac{\left<f,\phi_{j,\ell}^-(\alpha,\cdot)\right>}{\mu_{j,\ell}^-(\alpha)-k^2}\phi_{j,\ell}^-(\alpha,\cdot)\\&+\sum_{j=Q+1}^\infty\frac{\left<f,\phi_j(\alpha,\cdot)\right>}{\mu_j(\alpha)-k^2}\phi_j(\alpha,\cdot),
\end{aligned}
\]
where $\left<\cdot,\cdot\right>$ is the inner product in $L^2(\Omega_0)$.
\begin{figure}[ht]
	\centering
	\begin{tabular}{c  c}
		\includegraphics[width=0.5\textwidth,height=0.36\textwidth]{alpha_space.eps} 
		& \includegraphics[width=0.4\textwidth,height=0.36\textwidth]{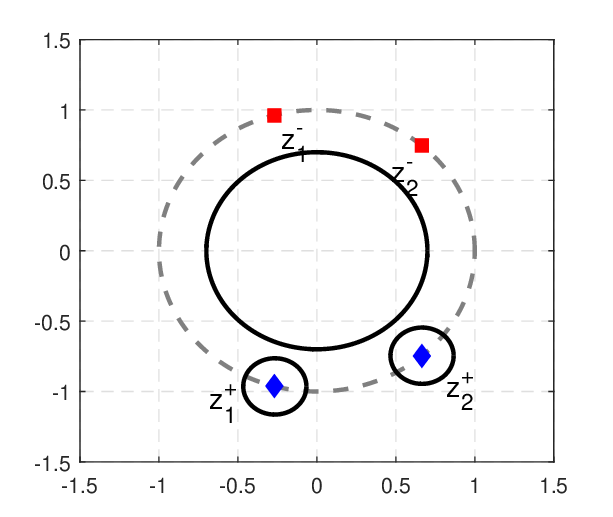}\\[-0cm]
	\end{tabular}
	\caption{Left: dispersion diagram; right: the new integral curve.}
	\label{fig:integral_curve_decomp}
\end{figure}

\high{We define, for sufficiently small $\delta>0$, the integral curve $\mathcal{C}_0$ as in Definition \ref{def:C_0}.} Then from Theorem \ref{th:LAP} and Cauchy's integral formula,
\[
\begin{aligned}
 u(x_1,x_2)&=\frac{1}{2\pi\i}\oint_{\mathcal{C}_0}w(z,x)z^{-1}\d z\\
 &=\frac{1}{2\pi\i}\oint_{|z|=\exp(-\tau)} w(z,x)z^{-1}\d z\\
 &\qquad+\sum_{j=1}^Q \left[\frac{1}{2\pi\i}\oint_{|z-z_j^+|=\delta} w(z,x)z^{-1}\d z\right],\quad x\in\Omega_0\\
 \end{aligned}
\]
For the visualization of the integral curve we refer to Fig \ref{fig:integral_curve_decomp}. 
The first term is evanescent and we only consider the second term. For any fixed $j=1,2,\dots,Q$, let
\[
 u^+_j(x):=\frac{1}{2\pi\i}\oint_{|z-z_j^+|=\delta} w(z,x)z^{-1}\d z,\quad x\in\Omega_0.
\]
From the representation of \high{$\widetilde{w}(\alpha,x):=w(-\i\log(z),x)$}, as the transform $\alpha:=-\i\log(z)$ maps the ball $B(z_j^+,\delta)$ into a small neighourhood of $\alpha_j^+$, denoted by $\mathcal{N}(\alpha_j^+,\delta)$, then $u_j$ becomes
\[
 u_j^+(x)=\frac{1}{2\pi}\oint_{\partial \mathcal{N}(\alpha_j^+,\delta)}\high{\widetilde{w}}(\alpha,x)\d\alpha.
\]
Note that the neighbourhood $\mathcal{N}(\alpha_j^+,\delta)$ does not contain any other points in $P_\pm$ if $\delta>0$ is sufficiently small. Then from the representation of $\high{\widetilde{w}}(\alpha,x)$, only the $j$-th term in the first series has a pole in $\mathcal{N}(\alpha_j^+,\delta)$. From the representation of the resolvent and residue theorem,
\[
 u_j^+(x)=\i{\rm Res}\left(\high{\widetilde{w}}(\alpha,x),\alpha=\alpha_j^+\right)=\sum_{\ell=1}^{L_j}u_{j,\ell}^+,\]
 where\[ u_{j,\ell}^+=\frac{\i\left<f,\phi_{j,\ell}^+(\alpha_j^+,\cdot)\right>}{(\mu_{j,\ell}^+)'(\alpha_j^+)}\phi_{j,\ell}^+(\alpha_j^+,\cdot).
\]
Note that $u_{j,\ell}^+$ is exactly the $\ell$-th Bloch wave  \high{solution corresponding to the Floquet multiplier $z_j^+$ propagating} to the right.

From the analysis above, the LAP solution can also be written as the decomposition of an evanescent function and $Q$ propagating modes:
\[
\begin{aligned}
u(x_1,x_2)=&\frac{1}{2\pi\i}\oint_{|z|=\exp(-\tau)} w(z,x)z^{-1}\d z\\&\qquad+\i\left[\sum_{j=1}^Q\sum_{\ell=1}^{L_j}\frac{\left<f,\psi_{j,\ell}^+(\alpha_j^+,\cdot)\right>}{(\mu_{j,\ell}^+)'(\alpha_j^+)}\psi_{j,\ell}^+(\alpha_j^+,\cdot)\right],\quad x\in\Omega_0.
\end{aligned}
\]
We can also extend the representation to $\Omega_n$ where $n\in\Z$:
\begin{equation}
\label{eq:LAP_sep}
u(x_1+n,x_2)=\begin{cases}
\displaystyle
\begin{aligned}&  \i\left[\sum_{j=1}^Q\sum_{\ell=1}^{L_j}\frac{\exp(\i n\alpha_j^+)\left<f,\psi_{j,\ell}^+(\alpha_j^+,\cdot)\right>}{(\mu_{j,\ell}^+)'(\alpha_j^+)}\psi_{j,\ell}^+(\alpha_j^+,\cdot)\right]      \\&\,+
\frac{1}{2\pi\i}\oint_{|z|=\exp(-\tau)} w(z,x)z^{n-1}\d z,\qquad   \qquad\qquad\qquad n\geq 0;
\end{aligned}\\
\\
\displaystyle
\begin{aligned}    &  \i\left[\sum_{j=1}^Q\sum_{\ell=1}^{L_j}\frac{\exp(\i n\alpha_j^-)\left<f,\psi_{j,\ell}^-(\alpha_j^-,\cdot)\right>}{(\mu_{j,\ell}^-)'(\alpha_j^-)}\psi_{j,\ell}^-(\alpha_j^-,\cdot)\right]   \\&\,+\frac{1}{2\pi\i}\oint_{|z|=\exp(\tau)} w(z,x)z^{n-1}\d z,\qquad \qquad\qquad\qquad \quad  n< 0.\end{aligned}
\end{cases}
\end{equation}



\section{Numerical scheme}
\label{sec:numer_scheme}

In this section, we consider two numerical methods to approximate LAP solutions of \eqref{eq:wg1}-\eqref{eq:wg2}. The first method (CCI-method) is based on the complex curve integral \eqref{eq:u_LAP} and the second method (PM-method) is based on the propagating modes \eqref{eq:LAP_sep}. Both cases involve numerical integral of solutions $w(z,x)\big|_{\mathcal{C}_0}$ of quasi-periodic \eqref{eq:wg_z1}-\eqref{eq:wg_z2} (or equivalently, the periodic problem \eqref{eq:quasi_periodic_var}). As the quasi-periodic problems are well-posed, they can be solved by classic numerical methods.

We apply the finite element method to solve the quasi-periodic problem \eqref{eq:quasi_periodic_var}. Suppose $\Omega_0$ is covered by a family of regular and quasi-uniform meshes (see \cite{Brenn1994,Saute2007}) $\mathcal{M}_h$ with the largest mesh width $h_0>0$. To construct periodic nodal functions, we suppose that the nodal points on the left and right boundaries have the same height. Omitting the nodal points on the right boundary, let $\left\{\phi_M^{(\ell)}\right\}$ be the piecewise linear nodal functions that equals to one at the $\ell$-th nodal and zero at other nodal points, then it is easily extended into a globally continuous and periodic function in $H^1_{loc}(\Omega)$. Define the discretization subspace by:
\begin{equation*}
V_h:={\rm span}\left\{\phi_M^{(1)},\phi_M^{(2)},\dots,\phi_M^{(M)}\right\}\subset H^1_\p(\Omega_0).
\end{equation*}
Thus we are looking for the finite element solution to \eqref{eq:quasi_periodic_var} with the expansion 
\begin{equation*}
v_z^h(x)=\sum_{\ell=1}^M c_z^\ell\phi_m^{(\ell)}(x)
\end{equation*}
satisfies
\begin{equation*}
a_z(v_z^h,\phi_h)=-\int_{\Omega_0}z^{-x_1}f\overline{\phi_h}\d x,\quad\forall\,\phi_h\in V_h.
\end{equation*}

By Theorem 14 in \cite{Lechl2016a} the finite element approximation is estimated as follows.

\begin{theorem}\label{th:est_v_z}
	Suppose $f\in L^2(\Omega_0)$ and $q\in W^{1,\infty}(\Omega_0)$, the solution $v_z\in H^2(\Omega_0)$ and the error between $v_z^h$ and $v_z$ is bounded by
	\begin{equation*}
	\|v_z^h-v_z\|_{L^2(\Omega_0)}\leq Ch^2\|f\|_{L^2(\Omega_0)},\quad \|v_z^h-v_z\|_{H^1_\p(\Omega_0)}\leq Ch\|f\|_{L^2(\Omega_0)}
	\end{equation*}
	where $C>0$ is a constant independent of $z\in\mathcal{C}_0$. The estimations are also true for the function $w(z,x)=(\F u)(z,x)$ defined in Section 4.2 when $z\in\mathcal{C}_0$, i.e.,
	\begin{eqnarray*}
		&& \|w_h(z,\cdot)-w(z,\cdot)\|_{L^2(\Omega_0)}\leq Ch^2\|f\|_{L^2(\Omega_0)};\\
		&& \|w_h(z,\cdot)-w(z,\cdot)\|_{H^1_\p(\Omega_0)}\leq Ch\|f\|_{L^2(\Omega_0)}.
	\end{eqnarray*}
\end{theorem}
In both of the following methods, $w_h(z,\cdot)$ is the numerical approximation of the  quasi-periodic solution $w(z,\cdot)$.


\subsection{CCI-method}
\label{sec:cci}

In this section, we consider numerical approximation of the curve integral  \eqref{eq:u_LAP}. $w(z,\cdot)$ is supposed to be known as it is easily computed by standard numericl methods. We can alway choose different $n$'s such that computations are carried out in different $\Omega_n$'s, but in this section $n$ is fixed as $0$ as an example. As the curve integral on $\UC$ is standard, we only consider the \high{case $k^2\in\sigma(A)$}, so the LAP solution $u$ is written in the form of \eqref{eq:u_LAP}.

Recall that the number of different values of elements in $S_+^0$ is $Q$, so is it in $S_-^0$. 
Thus $\mathbb{UF}=S_+^0\bigcup S_-^0$ has $2Q$ different values when Assumption \ref{asp2} holds. Especially, from  Remark \ref{rm:asp2} and Assumption \ref{asp2}, neither $1$ nor $-1$ lies  in \high{$\mathbb{UF}$}. Then the curve $\mathcal{C}_0$ is parameterized as follows.

Suppose  \high{$\mathbb{UF}=\left\{\varepsilon_j:\, j=1,2,\dots,2Q\right\}$} and
\begin{equation*}
 -\pi<\varepsilon_1<\varepsilon_2<\cdots <\varepsilon_{2Q-1}<\varepsilon_{2Q}<\pi.
\end{equation*}
By assumption, for any $j=1,\dots,2Q$, there is a sufficiently small $\delta>0$ such that the ball $B\big(\exp(\i\varepsilon_j),\delta\big)$ does not contain any other poles except for $\exp(\i\varepsilon_j)$ itself. Let $\varepsilon_j^-<\varepsilon_j^+$ be two angles such that $\left\{\exp(\i\varepsilon_j^+),\exp(\i\varepsilon_j^-)\right\}=\partial B\big(\exp(\i\varepsilon_j),\delta\big)\cap \UC$. From direct calculation,
\begin{equation}
\high{ \varepsilon_j^-=\varepsilon_j-2\arcsin\left(\delta/2\right),\quad \varepsilon_j^+=\varepsilon_j+2\arcsin\left(\delta/2\right).}
\end{equation}
Note that we choose $\delta$ and $\delta$ sufficiently small such that $\varepsilon_1^->-\pi$ and $\varepsilon_{2Q}^+<\pi$. Moreover,
\begin{equation}
 -\pi<\,\varepsilon_1^-<\varepsilon_1<\varepsilon_1^+\,<\,\varepsilon_2^-<\varepsilon_2<\varepsilon_2^+\,<\,\cdots\,<\,\varepsilon_{2Q}^-<\varepsilon_{2Q}<\varepsilon_{2Q}^+\,<\,\pi.
\end{equation}
Thus the curve segment of $\UC\cap\mathcal{C}_0$ is parameterized as
\begin{equation*}
 \mathcal{C}_j:=\left\{\exp(\i\varepsilon):\,\varepsilon\in[\varepsilon_j^+,\varepsilon_{j+1}^-]\right\},\quad j=1,2,\dots,2Q,
\end{equation*}
 where $\varepsilon_{2Q+1}^-:=\varepsilon_1^-+2\pi$.
For any $j\in \{1,2,\dots,2Q\}$,  $\mathcal{C}_0\cap\partial B\big(\exp(\i\varepsilon_j),\delta\big)$ lies in the exterior when $\exp(\i\varepsilon_j)\in S_+^0$; while $\mathcal{C}_0\cap\partial B\big(\exp(\i\varepsilon_j),\delta\big)$ lies in the interior when $\exp(\i\varepsilon_j)\in S_-^0$. Let this curve segment be represented by $\exp(\i\varepsilon_j)+\delta\exp(\i\theta)$. From direct calculation, there are two angles  $\theta_j^\pm$ such that $\left\{\exp(\i\varepsilon_j^+),\exp(\i\varepsilon_j^-)\right\}=\left\{\exp(\i\varepsilon_j)+\delta\exp(\i\theta_j^-),\exp(\i\varepsilon_j)+\delta\exp(\i\theta_j^+)\right\}$.
By choosing proper branches of the logarithmic function, $\theta_j^\pm$ satisfy
\begin{equation*}
 \theta_j^-<\theta_j^+<\theta_j^-+2\pi,
\end{equation*}
such that
\begin{eqnarray*}
 && \big|\exp(\i\varepsilon_j)+\delta\exp(\i\theta)\big|>1,\,\text{ when }\exp(\i\varepsilon_j) \in S_+^0;\\
 && \big|\exp(\i\varepsilon_j)+\delta\exp(\i\theta)\big|<1,\,\text{ when }\exp(\i\varepsilon_j) \in S_-^0.
\end{eqnarray*}
Thus the curve segment $\mathcal{C}_0\cap\partial B(\exp(\i\varepsilon_j),\delta)$ is parameterized as
\begin{equation*}
 \mathcal{D}_{j}=\exp(\i\varepsilon_j)+\delta\exp(\i\varepsilon),\quad \theta\in(\theta_j^-,\theta_j^+),\quad j=1,2,\dots,2Q.
\end{equation*}

Now the whole curve $\mathcal{C}_0$ 
is parametrized piecewisely. Thus the representation of $u$ in \eqref{eq:u_LAP} is written as
\begin{align*}
u(x)&=\frac{1}{2\pi\i}\sum_{j=1}^{2Q}\int_{\mathcal{C}_j}w(z,x)z^{-1}\d z+\frac{1}{2\pi\i}\sum_{j=1}^{2Q}\int_{\mathcal{D}_j}w(z,x)z^{-1}\d z\\
&=\frac{1}{2\pi}\int_{\alpha_j^+}^{\alpha_{j+1}^-}w(\exp(\i t),x)\d t\\
&+\frac{1}{2\pi}\int_{\theta_j^-}^{\theta_j^+}w(\exp(\i\alpha_j)+\delta\exp(\i t),x)\frac{\delta\exp(\i t)}{\exp(\i\alpha_j)+\delta\exp(\i t)}\d t
\end{align*}

All of the integrands depend smoothly on $t$. Thus we only need to consider the numerical integration
\begin{equation*}
 \I(g):=\int_a^b g(t,x)\d t,
\end{equation*}
where $a<b$ and $g$ depends analytically on $t\in(a,b)$ and $g(t,\cdot)\in H^2(\Omega_0)$ for any fixed $t$. 


For an efficient numerical integral, we adopt the method introduced in \cite{Zhang2017e}, which comes originally from Section 3.5, \cite{Colto1998}. Let $q(\tau):\,(a,b)\rightarrow [-\pi,\pi]$ be a smooth and strictly monotonically increasing function and satisfies 
\begin{equation}\label{eq:q}
 q(-\pi)=a;\,q(\pi)=b;\,q^{(\ell)}(-\pi)=q^{(\ell)}(\pi)=0,\,\ell=1,2,\dots,N_0
\end{equation}
for some positive integer $N_0$. Let $t=q(\tau)$, the integral $\I(g)$ becomes
\begin{equation*}
 \I(g)=\int_{-\pi}^{\pi} g(q(\tau),x)q'(\tau)\d \tau:=\int_{-\pi}^{\pi} s(\tau,x)\d\tau.
\end{equation*}
Thus the new integrand $s$ depends smoothly on $\tau$ and is extended to a periodic function with respect to $\tau$. 

We approximate the integral $\I(g)$ by trapezoidal rule. Let $[-\pi,\pi]$ be divided uniformly into $N$ 
subintervals, and the grid points be
\begin{equation*}
 t_j=-\pi+\frac{2\pi}{N}j,\quad j=1,\dots,N.
\end{equation*}
The integral is approximated by
\begin{equation}\label{eq:quad}
\I_N(g)=\frac{2\pi}{N}\sum_{\ell=1}^N s(t_\ell,x)=\frac{2\pi}{N}\sum_{\ell=1}^N g(q(t_\ell),x)q'(t_\ell).
\end{equation}
Then we recall the result in \cite{Zhang2017e}, and obtain the error estimation of the integral via \eqref{eq:quad};
 \begin{equation}\label{eq:err_quad}
  \left\|\I(g)-\I_N(g)\right\|_{H^2(\Omega_0)}\leq CN^{-N_0+1/2}\|s\|_{C^{N_0}([-\pi,\pi];S(W))}.
 \end{equation}

We apply the quadrature rule \eqref{eq:quad} to approximate $u(x)$:
\begin{equation}\label{eq:appx_uN}
\begin{aligned}
 &u_N(x)=\frac{2\pi}{N}\sum_{j=1}^{2Q}\left[\sum_{\ell=1}^N w(\exp(\i q_j(t_\ell)),x)q'_j(t_\ell)\right.\\
 &\left.+\sum_{\ell=1}^N w(\exp(\i\alpha_j)+\delta\exp(\i q_{j+2Q}(t_\ell)),x)\frac{\delta q'_{j+2Q}(t_\ell)\exp(\i q_{j+2Q}(t_\ell))}{\exp(\i\alpha_j)+\delta\exp(\i q_{j+2Q}(t_\ell))}\right],
 \end{aligned}
\end{equation}
where  $q_j$ be the smooth function from $[\alpha_j^+,\alpha_{j+1}^-]$ to $[-\pi,\pi]$ and $q_{j+2Q}$ are the smooth function from $[\theta_j^-,\theta_{j}^+]$ to $[-\pi,\pi]$.
With \eqref{eq:err_quad}, we conclude that
\begin{equation}\label{eq:est_uN}
 \left\|u(x)-u_N(x)\right\|_{H^2(\Omega_0)}\leq CN^{-N_0+1/2}\|w\|_{C^{N_0}([-\pi,\pi];S(W))}.
\end{equation}


Finally, we replace $w(z,x)$  by the finite element solution $w_h(z,x)$. Then  the LAP solution is approximated by
\begin{equation}\label{eq:app_Nh}
\begin{aligned}
 u_{N,h}(x)&=\frac{1}{N}\sum_{j=1}^{2Q}\left[\sum_{\ell=1}^N w_h(\exp(\i q_j(t_\ell)),x)q'_j(t_\ell)\right.\\
 &\left.+\sum_{\ell=1}^N w_h(\exp(\i\alpha_j)+\delta\exp(\i q_j(t_\ell)),x)\frac{\delta q'_{j+2Q}(t_\ell)\exp(\i q_j(t_\ell))}{\exp(\i\alpha_j)+\delta\exp(\i q_j(t_\ell))}\right].
 \end{aligned}
\end{equation}
Then we conclude the error between $u_{N,h}$ and $u$ in the domain $\Omega_0$.

\begin{theorem}\label{th:err}
Let $u_{N,h}$ be the numerical solution comes from the finite element method and the integral approximation \eqref{eq:app_Nh}.  Then the error is bounded by
 \begin{equation}
 \begin{aligned}
 & \|u_{N,h}-u\|_{L^2(\Omega_0)}\leq C\left(N^{-N_0+1/2}+h^2\right)\|f\|_{L^2(\Omega_0)};\\& \|u_{N,h}-u\|_{H^1(\Omega_0)}\leq C\left(N^{-N_0+1/2}+h\right)\|f\|_{L^2(\Omega_0)},
  \end{aligned}
 \end{equation}
where $C$ is a constant that depends on $Q$ and $N_0$, but does not depend on $N$ and $h$. 
\end{theorem}

\begin{proof}
From the representations of $u_{N,h}$ and $u$, and also the results from \eqref{eq:est_uN}, we have the following error estimation:
 \begin{equation*}
  \begin{aligned}
   &\,\,\,\,\,\,\,\,\|u_{N,h}-u\|_{L^2(\Omega_0)}\\
   &\leq \|u_{N,h}-u_N\|_{L^2(\Omega_0)}+\|u_N-u\|_{L^2(\Omega_0)}\\
   &\leq \left\|\frac{1}{N}\sum_{j=1}^{2Q}\left[\sum_{\ell=1}^N (w_h-w)(\exp(\i q_j(t_\ell)),x)q'_j(t_\ell)\right.\right.\\
 &\left.\left.+\sum_{\ell=1}^N (w-w_h)(\exp(\i\alpha_j)+\delta\exp(\i q_j(t_\ell)),x)\frac{\delta q'_{j+2Q}(t_\ell)\exp(\i q_j(t_\ell))}{\exp(\i\alpha_j)+\delta\exp(\i q_j(t_\ell))}\right]\right\|_{L^2(\Omega_0)}\\
 &\qquad+CN^{-N_0+1/2}\|f\|_{H^1(\Omega_0)}\\
 &\leq \frac{C}{N}\sum_{j=1}^{2Q}\sum_{\ell=1}^N\left[\|(w-w_h)(\exp(\i\alpha_j)+\delta\exp(\i q_j(t_\ell)),\cdot)\|_{L^2(\Omega_0)}\right.\\
 &\qquad\left.+\|(w-w_h)(\exp(\i\alpha_j)+\delta\exp(\i q_j(t_\ell)),\cdot)\|_{L^2(\Omega_0)}\right]
  +CN^{-N_0+1/2}\|f\|_{H^1(\Omega_0)}\\
  &\leq C\left(h^2+N^{-N_0+1/2}\right)\|f\|_{H^1(\Omega_0)}.
  \end{aligned}
 \end{equation*}
The estimation of the $H^1-$norm is also obtained in the same way:
\begin{equation*}
 \|u_{N,h}-u\|_{H^1(\Omega_0)}\leq C\left(h+N^{-N_0+1/2}\right)\|f\|_{H^1(\Omega_0)}.
\end{equation*}
The proof is finished.
 
\end{proof}

\subsection{PM-method}\label{sec:pm}

In this subsection, we introduce another numerical method based on \eqref{eq:LAP_sep}. The integration part in \eqref{eq:LAP_sep} can be approximated in the same way as the CCI method, then we only need to deal with the first term. 

The first step is to find out explicit values of $\alpha_j^\pm$ and corresponding eigenfunctions $\psi^\pm_{j,\ell}(\alpha_j^\pm,\cdot)$. From the variational formulation \eqref{eq:quasi_periodic_var}
and by replacing $\log(z)$ by $\i\alpha$ in \eqref{eq:quasi_periodic_var}, the problem is to find $\alpha\in (-\pi,\pi]$ such that there is a non-trival $v\in H^1_\p(\Omega_0)$ solving
\begin{equation*}
\int_{\Omega_0}\left[\grad v\cdot\grad\overline{\phi}+\i\alpha\left(v\frac{\partial\overline{\phi}}{\partial x_1}-\frac{\partial v}{\partial x_1}\overline{\phi}\right)-(k^2 q-\alpha^2)v\overline{\phi}\right]\d x
=0,
\end{equation*}
for any $\phi\in H^1_\p(\Omega_0)$. 
 From Theorem 2.4, \cite{Kirsc2017a}, it is written as a generalized eigenvalue problem:
 \begin{equation}\label{eq:gep}
 \left(
 \begin{matrix}
 I-K_0 & 0 \\ 0 & I
 \end{matrix}
 \right)
  \left(
 \begin{matrix}
 u_1 \\ u_2
 \end{matrix}
 \right)=-\alpha \left(
 \begin{matrix}
 B & C^{1/2} \\ -C^{1/2} & 0
 \end{matrix}
 \right)
 \left(
 \begin{matrix}
 u_1 \\ u_2
 \end{matrix}
 \right).
 \end{equation}
By solving this problem, we can obtain the values of $\alpha_j^\pm$ and the corresponding eigenfunctions 
$\psi^\pm_{j,\ell}(\alpha_j^\pm,\cdot)$ at the same time.  Then the values of $\left<f,\psi^\pm_{j,\ell}(\alpha_j^\pm,\cdot)\right>$
are easily calculated by numerical integral on triangular meshes.

Now we have to evaluate $(\mu_{j,\ell}^\pm)'(\alpha_j^+)$. First we draw the dispersion diagram in the neighbourhood of each $\alpha_j^+$, then the derivative can be computed by the symmetric difference quotient, i.e., for a sufficiently small $\delta_0>0$,
\[
(\mu_{j,\ell}^\pm)'(\alpha_j^+)\approx \frac{(\mu_{j,\ell}^\pm)(\alpha_j^+ +\delta_0)-(\mu_{j,\ell}^\pm)(\alpha_j^+ -\delta_0)}{2\delta_0}.
\]
The errors brought by solutions of generalized eigenvalue problems, numerical integration and numerical differentiation are sufficiently small, thus are omitted. Thus the main error still comes from the first term, and can be estimated by Theorem \ref{th:err}.

\section{Half-guide problems}
\label{sec:half-guide}

The numerical method introduced in this paper is also extended to  half-guide scattering problems. Let the half waveguide be defined by $\Omega_+:=\bigcup_{n=1}^\infty\Omega_n$ (see Figure \ref{half-waveguide}). Then we are looking for the LAP solution $u_+\in H_{loc}^1(\Omega_+)$ such that it satisfies
\begin{equation}\label{eq:half-guide}
\Delta u_++k^2 qu_+=0\quad\text{ in }\Omega_+;\quad u_+=\phi\quad\text{ on }\Gamma_1.
\end{equation}
From \cite{Hoang2011}, for any $\phi\in H^{1/2}(\Gamma_1)$ the LAP solution exists if Assumption \ref{asp1} holds.
In this section, we introduce numerical methods to approximate those LAP solutions with Assumption \ref{asp1} and \ref{asp2}. 

\begin{figure}[ht]
\centering
\includegraphics[width=8cm]{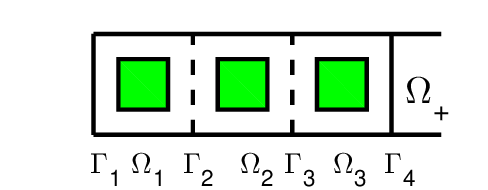}
\caption{Periodic waveguide.}
\label{half-waveguide}
\end{figure}

\subsection{Approximation of LAP solutions}

First, we define the following space:
\[
\mathcal{U}:=\left\{u(f)\big|_{\Gamma_1}:\,u(f)\in H^1_{loc}(\Omega) \text{ is the LAP solution of \eqref{eq:wg1}-\eqref{eq:wg2} with }f\in L^2(\Omega)\right\}.
\]
From the analysis in \cite{Zhang2019a}, $\overline{\mathcal{U}}=H^{1/2}(\Gamma_1)$. For any fixed $\phi\in H^{1/2}(\Gamma_1)$, given any sufficiently small $\epsilon_0>0$, there is an $\Omega_0$-supported function $f$ such that $\|u(f)\big|_{\Gamma_1}-\phi\|_{H^{1/2}(\Gamma_1)}\leq\epsilon_0$. Then $u_+$ is approximated by:
\begin{equation}\label{eq:LAP_half}
 \widetilde{u}_+(x_1+n,x_2)=\frac{1}{2\pi\i}\oint_{\mathcal{C}_0}w(z,x)z^{n-1}\d z,\quad x\in\Omega_0,\,n\in\N
\end{equation}
where $w(z,\cdot)$ solves \eqref{eq:quasi1}-\eqref{eq:quasi2} with the right hand side $f$, and $\mathcal{C}_0$ is defined by Definition \ref{def:C_0}. 

To solve this problem with the method based on the approximation \eqref{eq:LAP_half}, a two-step method is developed. The first step use the density of $\mathcal{U}$ in $H^{1/2}(\Gamma_1)$ to determine a source term $f$. The second step use the source term to compute $u_+$, with the same method as the full-guide problem. As the second step is exactly the same as the full-guide problem, we only discuss the first step in this section.

Let $A$ be the operator from $L^2(\Omega_0)$ to $H^{1/2}(\Gamma_1)$ defined by
\begin{equation}
f\mapsto \left.\left(\frac{1}{2\pi\i}\oint_{\mathcal{C}_0}w(z,x)z^{-1}\d z\right)\right|_{\Gamma_1}.
\end{equation}
From Theorem \ref{th:LAP}, $u(f)\in H^2(\Omega_0)$. Thus 
$A$ is a compact operator with a dense range. We are looking for an $f$ such that
\begin{equation*}
 Af=\phi\quad\text{ on }\Gamma_1.
\end{equation*}
As the equation is severely ill-posed, we apply the Tikhonov regularization method to find the ``best solution'' of this equation.

\subsection{Tikhonov regularization}

 Given a sufficiently small regularization parameter $\gamma>0$, we are looking for an $f_\gamma\in L^2(\Omega_0)$ such that
\begin{equation*}
 (A^*A+\gamma I)f_\gamma=A^*\phi.
\end{equation*}
Suppose $\{\mu_n,u_n,v_n\}$ is the singular system of the operator $A$, and $\mu_1\geq\mu_2\geq\cdots\geq\mu_n\geq\cdots>0$. From the definition,
\begin{equation*}
 A u_n=\mu_n v_n,\quad A^* v_n=\mu_n u_n.
\end{equation*}
For any $\phi\in H^{1/2}(\Gamma_1)$, there is a series $\{c_n:\,n\in \N\}\subset\C$ such that
\begin{equation*}
 \phi=\sum_{n\in \N}c_n v_n,\quad\text{ where     }\|\phi\|_{H^1(\Omega_0)}^2=\sum_{n\in\N}|c_n|^2.
\end{equation*}
From direct calculation, 
\begin{equation*}
 f_\gamma=\sum_{n\in\N}\frac{\mu_n}{\mu_n^2+\gamma}c_n u_n.
\end{equation*}

Then we estimate the error between $A f_\gamma$ and $\phi$:
\begin{equation*}
 \|A f_\gamma-\phi\|^2=\left\|\sum_{n\in\N}\frac{\mu_n^2}{\mu_n^2+\gamma}c_n v_n-\sum_{n\in\N}{c_n} v_n\right\|^2=\sum_{n\in\N}\frac{\gamma^2 |c_n|^2}{(\mu_n^2+\gamma)^2}.
\end{equation*}

As $\sum_{n\in\N}|c_n|^2$ converges, given any $\delta>0$, there is an integer $L>0$ such that
\begin{equation*}
 \sum_{n=L+1}^\infty {|c_n|^2}<\delta/2
 \quad\implies\quad
 \sum_{n=L+1}^\infty\frac{\gamma^2 |c_n|^2}{(\mu_n^2+\gamma)^2}<\sum_{n=L+1}^\infty {|c_n|^2}<\delta/2.
\end{equation*}
For $n=0,1,\dots,L$, let $\displaystyle\gamma<\min_{n=0,1,\dots,L}\sqrt{\frac{\delta\mu_n^4}{(2L+2)|c_n|^2}}$, then
\begin{equation*}
 \frac{\gamma^2|c_n|^2}{(\mu_n^2+\gamma)^2}\leq \frac{\gamma^2|c_n|^2}{\mu_n^4}<\delta/(2L+2).
 \end{equation*}
Thus
\begin{equation*}
 \|A f_\gamma-\phi\|^2\leq \delta/2+(L+1)\delta/(2L+2)=\delta.
\end{equation*}
This implies that when $\gamma\rightarrow 0$, $A f_\gamma\rightarrow \phi$ in $H^{1/2}(\Gamma_1)$. Thus the corresponding solution also converges as $\gamma\rightarrow 0$. 

\subsection{Numerical method}

Now we conclude the numerical scheme for half-guide scattering problems as follows.

\begin{enumerate}
 \item Choose a family of basis functions for the space $H^1(\Omega_0)$ and denote it by $\{\phi_\ell\}_{\ell=1}^\infty$. Fix a large enough integer $M_0>0$ and construct the matrix
\begin{equation*}
 \Phi:=\left(A\phi_1,A\phi_2,\dots,A\phi_{M_0}\right).
\end{equation*}
\item Fix a small regularization parameter $\gamma>0$, compute the vector
\begin{equation*}
 c(\gamma)=\left[\Phi^*\Phi+\gamma I\right]^{-1}\left(\Phi^*\phi\right).
\end{equation*}
\item Approximate the function $f$ by $f_\gamma^{M_0}$ with coefficients $c(\gamma)$:
\begin{equation*}
 f_\gamma^{M_0}(x)=\sum_{\ell=1}^{M_0}c_\ell(\gamma)\phi_\ell(x).
\end{equation*}
\item Solve the equation \high{\eqref{eq:wg_z1}-\eqref{eq:wg_z2}} with the source term $f_\gamma^{M_0}$. The LAP solution $u$ is then obtained by either the CCI method or the PM method. 
\end{enumerate}

From the convergence of the numerical method to approximate $f$ and solve \eqref{eq:LAP_half}, the full numerical scheme converges as $M_0\rightarrow \infty$, $\gamma\rightarrow 0$, $h\rightarrow0$ and $N\rightarrow\infty$.

\section{Special wavenumbers}
\label{sec:special}

In previous sections are discussed based on Assumption \ref{asp2}, i.e., $S_+^0\cap S_-^0=\emptyset$. As it is not a necessary condition for the limiting absorption principle, we consider the case  $S_+^0\cap S_-^0\neq\emptyset$ in this section. Recall that 
\begin{align*}
& S_+^0=\left\{z_{j,\ell}^+:\,j=1,\dots,Q;\,\ell=1,\dots,L_j\right\};\\
&S_-^0=\left\{z_{j,\ell}^-:\,j=1,\dots,Q;\,\ell=1,\dots,L_j\right\};
\end{align*}
where
\[
z_{j,1}^+=\cdots=z_{j,L_j}^+=z_j^+,\quad z_{j,1}^-=\cdots=z_{j,L_j}^-=z_j^-,\text{ for all }j=1,\dots, Q.
\]
Recall \eqref{eq:LAP_sep}, $u$ is written as
\[
 u(x)=\frac{1}{2\pi\i}\oint_{|z|=\exp(-\tau)} w(z,x)z^{n-1}\d z+\sum_{j=1}^Q\sum_{\ell=1}^{L_j} u_{j,\ell}^+(x),\quad x\in\Omega_0,
\]
where
\[
 u_{j,\ell}^+=\frac{\i\left<f,\psi_{j,\ell}^+(\alpha_j^+,\cdot)\right>}{(\mu_{j,\ell}^+)'(\alpha_j^+)}\psi_{j,\ell}^+(\alpha_j^+,\cdot).
\]
Note that $z_j^\pm=\exp(\i\alpha_j^\pm)$ for $j=1,2,\dots,Q$. Reorder the points $z_j^\pm$ and $\alpha_j^\pm$ such that $z_j^+=z_j^-$ for $j=1,2,\dots,Q'$ where $1\leq Q'\leq Q$, then $\alpha_j^+=\alpha_j^-$ for these $j$'s. Let $z_j:=z_j^\pm$ and $\alpha_j=\alpha_j^\pm$ for $j=1,2,\dots,Q'$. 


As $\left(\mu_{j,\ell}^+\right)'(\alpha_j^+)>0$ ($\ell=1,\dots,L_j$), its inverse function exists in a sufficiently small neighbourhood of $k^2_0$. Let the inverse function be denoted by $\eta_{j,\ell}^+$, then $\eta_{j,\ell}^+(k_0^2)=\alpha_j^+$. As $\mu_{j,\ell}^+$ depends analytically on $\alpha$, $\eta_{j,\ell}^+$ is an analytic function in a sufficiently small neighbourhood of $k^2_0$. Moreover, $\left(\eta_{j,\ell}^+\right)'(k^2_0)=\left(\left(\mu_{j,\ell}^+\right)'(\alpha_j^+)\right)^{-1}>0$, thus $\eta_{j,\ell}^+$ is a strictly increasing function near $k_0^2$. Similarly, the inverse function of $\mu_{j,\ell}^-$, denoted by $\eta_{j,\ell}^-$, also exists. The function $\eta_{j,\ell}^-$ is analytic in a sufficiently small neighbourhood of $k^2_0$, and is strictly decreasing.  For any $j=1,2,\dots,Q$, there is a sufficiently small neighbourhood of $k^2_0$, denoted by $K$, such that $\eta_{j,\ell}^+$ ($\eta_{j,\ell}^-$) exists and is analytic in $K$. Moreover, $\eta_{j,\ell}^+$ is strictly increasing and $\eta_{j,\ell}^-$ is strictly decreasing in $K$ for any $j=1,2,\dots,Q$. 

Let $K_+:=K\cap(k_0^2,\infty)$ and $K_-:=K\cap(0,k_0^2)$. When $k^2\in K_-$, for any $j=1,2,\dots,Q'$ and $\ell=1,\dots,L_j$,
\[\eta_{j,\ell}^+(k^2)<\eta_{j,\ell}^+(k_0^2)=\alpha_j,\quad \eta_{j,\ell}^-(k^2)>\eta_{j,\ell}^-(k_0^2)=\alpha_j.\]
Similarly, when $k^2\in K_+$, for any $j=1,2,\dots,Q'$ and $\ell=1,\dots,L_j$,
\[\eta_{j,\ell}^+(k^2)>\eta_{j,\ell}^+(k_0^2)=\alpha_j,\quad \eta_{j,\ell}^-(k^2)<\eta_{j,\ell}^-(k_0^2)=\alpha_j.\]
This implies that when $k^2\neq k_0^2$, the  unit Floquet multipliers $\mathbb{UF}(k^2)$ are separated. 
 In this case, $k$ is the wavenumber such that both Assumption \ref{asp1} and \ref{asp2} are satisfied, and the solution is computed 
by the algorithm introduced in previous sections. Let $u_{k^2}$ be the corresponding LAP solution, then from \eqref{eq:LAP_sep},
\[
 u_{k^2}(x)=\frac{1}{2\pi\i}\oint_{|z|=\exp(-\tau)} w_{k^2}(z,x)z^{n-1}\d z+\sum_{j=1}^Q\sum_{\ell=1}^{L_j} u_{j,\ell,k^2}^+(x),
\]
where
\[
 u_{j,\ell,k^2}^+=\frac{\i\left<f,\psi_{j,\ell}^+(\eta_{j,\ell}^+(k^2),\cdot)\right>}{(\mu_{j,\ell}^+)'(\eta_{j,\ell}^+(k^2))}\psi_{j,\ell}^+(\eta_{j,\ell}^+(k^2),\cdot).
\] 
It is clear that $u_{j,\ell,k^2}^+$ depends analytically on $k^2\in K$, where $K$ is a small neighbourhood of $k_0^2$. Let $k\rightarrow k_0$, the above equation still  holds.

 \begin{figure}[ht]
\centering
\begin{tabular}{c  c}
\includegraphics[width=0.5\textwidth,height=0.45\textwidth]{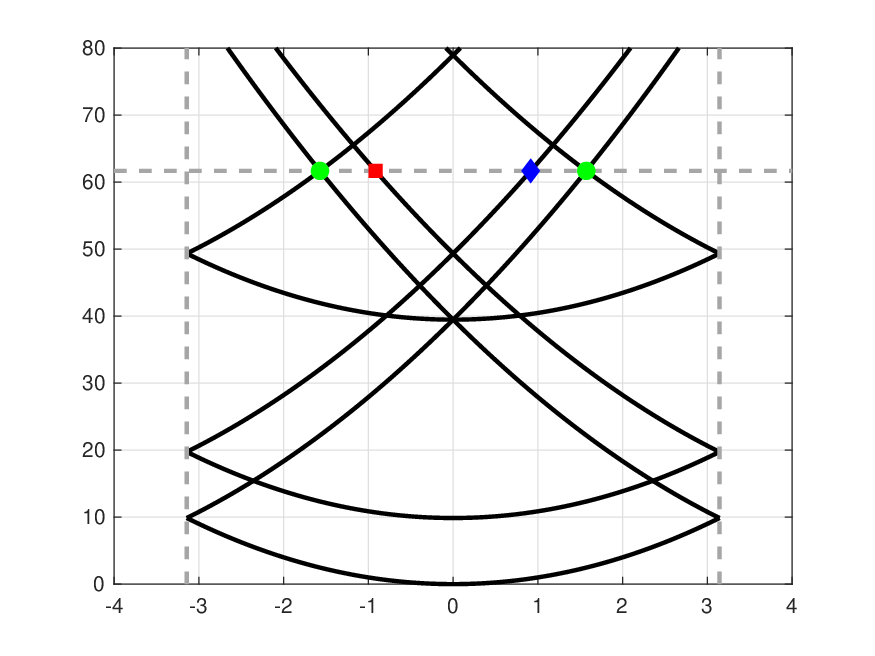} 
& \includegraphics[width=0.4\textwidth,height=0.45\textwidth]{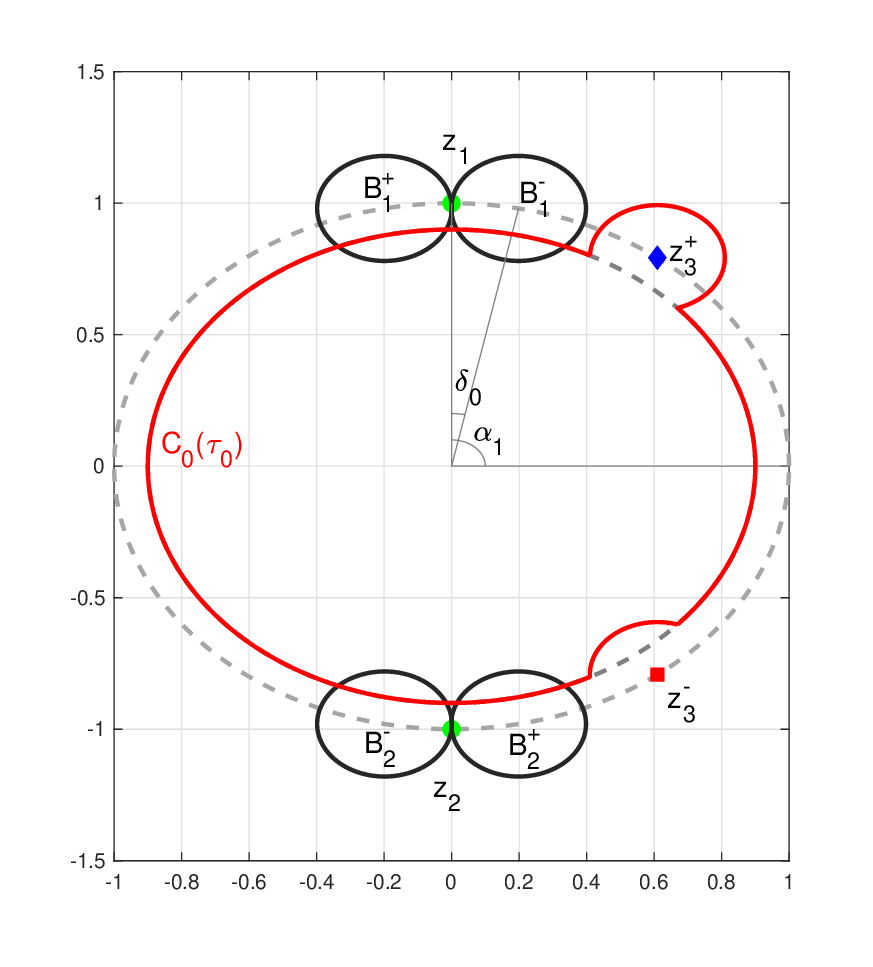}\\[-0cm]
\end{tabular}
\caption{Example for $n=1$ and $k^2_0\approx 61.685$. There are four unit Floquet multipliers, \high{corresponding} to $\alpha\approx \pm 1.5708,\,\pm 0.9151$. From the definition of $S_0^\pm$, $P'=2$ and $P=3$, the green circles are $z_1\approx \exp(1.5708\i)$ and $z_2\approx \exp(-1.5708\i)$, the red square is $z_3^-\approx \exp(- 0.9151\i)$ and the blue diamond is $z_3^+\approx \exp( 0.9151\i)$. The red curve is $\mathcal{C}_0(\tau_0)$, the balls with black boundaries are $B_j^\pm$ where $j=1,2$.}
\label{fig:curve_special}
\end{figure}

Let $\tau_0>0$ be sufficiently small such that the set $RS$ lies in the ball $B(0,\exp(-\tau_0))$. From Lemma \ref{th:s_pm_dist}, $\tau_0$ exists. Then we define the following domain:
\[B_{\tau_0,\delta}=B(0,\exp(-\tau_0))\bigcup\left[\bigcup_{n=Q'+1}^Q B(z_j^+,\delta)\right]\setminus \left[\bigcup_{n=Q'+1}^Q \overline{B(z_j^-,\delta)}\right],\]
where $\delta>0$ satisfies all the conditions in Definition \ref{def:C_0}. Let $\mathcal{C}_0(\tau_0)$ be the boundary of $B_{\tau_0,\delta}$ (see Figure \ref{fig:curve_special}). From the definition of $\mathcal{C}_0(\tau_0)$ and Theorem \ref{th:LAP}, the solution with respect to $k^2\in K\setminus\{k_0^2\}$ has the following representation:
\[u_{k^2}(x_1,x_2)=\frac{1}{2\pi\i}\int_{\mathcal{C}_0(\tau_0)}w_{k^2}(z,x)z^{-1}\d z+\sum_{j=1}^{Q'} \sum_{\ell=1}^{L_j}u_{j,\ell,k^2}^+(x).
\]
where
\[
u_{j,k^2}^+(x)=\sum_{\ell=1}^{L_j}u_{j,\ell,k^2}^+(x).
\]
When $z\in\mathcal{C}_0(\tau_0)$, the integrand $w_{k^2}(z,x)$ depends analytically on $k^2\in K$. As the first term is well defined when $k^2\rightarrow k_0^2$, we can easily prove that
\[\frac{1}{2\pi\i}\int_{\mathcal{C}_0(\tau_0)}w_{k^2}(z,x)z^{n-1}\d z\rightarrow \frac{1}{2\pi\i}\int_{\mathcal{C}_0(\tau_0)}w_{k^2_0}(z,x)z^{n-1}\d z\Big(:=u_0(x_1+n,x_2)\Big)\]
as $k^2\rightarrow k_0^2$.  
We only need to consider the numerical approximation of $u_{j,k_0^2}^+$ where $j=1,2,\dots,Q'$.

Let $\delta_0>0$ be a small number, \high{and define} the following two balls (see Figure \ref{fig:curve_special}) for any $j=1,2,\dots,Q'$:
\begin{equation}
\label{eq:ball}
\begin{aligned}
& B_j^-:=B\left(\exp\left[\i(\alpha_j-\delta_0)\right],2\sin\left(\frac{\delta_0}{2}\right)\right);\\
& B_j^+:=B\left(\exp\left[\i(\alpha_j+\delta_0)\right],2\sin\left(\frac{\delta_0}{2}\right)\right).
\end{aligned}
\end{equation}
Then the point $z_j=\exp(\alpha_j)\in\partial \,B_j^+\cap\partial\, B_j^-$. Furthermore, we suppose that $\delta_0$ is sufficiently small such that $B_j^\pm\cap \mathbb{F}=\emptyset$. Moreover, the following two \high{inclusions} are easily obtained when the neighbourhood $K$ is sufficiently small:
\begin{eqnarray*}
 && \left\{\exp\left(\i\eta_j^+(k^2)\right):\,k^2\in K_-\right\}\subset B_j^-;\quad \left\{\exp\left(\i\eta_j^-(k^2)\right):\,k^2\in K_-\right\}\subset B_j^+;\\
 && \left\{\exp\left(\i\eta_j^-(k^2)\right):\,k^2\in K_+\right\}\subset B_j^-;\quad \left\{\exp\left(\i\eta_j^+(k^2)\right):\,k^2\in K_+\right\}\subset B_j^+.
\end{eqnarray*}
Then for $j=1,2,\dots,Q'$, $u_{j,k^2}^+$ has the following integral representation:
\begin{equation}
\label{eq:u_j}
 u_{j,k^2}^+(x_1,x_2):=\begin{cases}
\frac{1}{2\pi\i}\int_{\partial B_j^-}w_{k^2}(z,x)z^{-1}\d z\quad\text{ when }k^2<k_0^2;\\
\frac{1}{2\pi\i}\int_{\partial B_j^+}w_{k^2}(z,x)z^{-1}\d z\quad\text{ when }k^2>k_0^2.
            \end{cases}
\end{equation}
As this function depends analytically on $k$, we can still apply the LAP to obtain the ``physically meaningful'' solution. The results are concluded in the following theorem.

\begin{theorem}
 \label{th:LAP_special}
 Suppose Assumption \ref{asp1} is satisfied, and $k_0$ is \high{a} wavenumber such that $S_0^+\cap S_0^-\neq\emptyset$. Let $\mathcal{C}_0(\tau_0)$ be defined as above, then 
 \[\lim_{k\rightarrow k_0}u_{k^2}(x_1,x_2)=u_{k_0^2}(x_1,x_2),\quad x\in\Omega_0.\]
 Moreover, $u_{k_0^2}$ has the \high{representation }
 \begin{equation}\label{eq:LAP_special}
  u_{k_0^2}(x_1,x_2)=\frac{1}{2\pi\i}\int_{\mathcal{C}_0(\tau_0)}w_{k^2_0}(z,x)z^{-1}\d z+\sum_{j=1}^{Q'}\lim_{k\rightarrow k_0} u_{j,k^2}^+,
 \end{equation}
 where $u_{j,k^2}^+$ is defined by \eqref{eq:u_j}.
\end{theorem}

\subsection{CCI method}

As was explained, the first term in \eqref{eq:LAP_special} can be evaluated directly by the numerical method introduced in \eqref{eq:appx_uN}. For the second term, an interpolation technique with respect to real valued $k^2$ is introduced to approximate the exact value. This is motivated by  \cite{Ehrhardt2009a}, where an extrapolation technique with respect to $\epsilon$ was developed.  We conclude the algorithm as follows. 
\begin{enumerate}
\item Compute $u_0$ from the method introduced in \eqref{eq:app_Nh}, denoted by $u_{0,N,h}$, where $N$ and $h$ are defined in the same way as in \eqref{eq:app_Nh}.
 \item Fix $M$ ($M\geq 2$) different wavenumbers $k_m^2\in K\setminus\{k_0^2\}$, $m=1,2,\dots,M$, evaluate the value $u_{k_m^2,j}^+$ for $j=1,2,\dots,Q'$. The values are approximated by \high{\eqref{eq:u_j}} with $N$ points, and denoted by $u_{j,k_m^2,N,h}^+$.
 \item With these $M$ points, we approximate the value at $k_0^2$ by interpolation from $u_{j,k_m^2,N,h}^+$ where $m=1,2,\dots,M$.
\item $u_{k_0^2,N,h}$ is then approximated by
\[ {u}_{k^2_0,N,M,h}(x_1+n,x_2)={u}_{0,N,h}+\sum_{j=1}^{Q'} {u}_{j,k_0^2,N,M,h}^+.\]
\end{enumerate}

The following error estimations can be obtained in the same way \high{as in} Theorem \ref{th:err}:
\[
\begin{aligned}
\|u_{0,N,h}-u_0\|_{H^\ell(\Omega_0)},\,&\left\|u_{j,k^2,N,h}^+-u_{j,k^2}^+\right\|_{H^\ell(\Omega_0)}\\&\leq C\left(N^{-N_0+1/2}+h^{2-\ell}\right)\|f\|_{L^2(\Omega_0)},\quad \ell=0,1.\end{aligned}\]
Note that $C$ does not depend on $N$ and $h$, but it depends on $k^2$ and may blow up when $k^2$ is close to $k_0^2$. To obtain the error between $u_{k^2_0,N,h,M}$ and $u_{k_0^2}$, we only need to estimate the error from the interpolation. As $u_{j,k^2}^+$ depends analytically on $k^2\in K$, there is a point $K^*\in K$ such that 
\[u_{j,k^2}^+=\sum_{\ell=0}^\infty u_\ell^j (k^2-K^*)^\ell,\]
and there is a $C_0>0$ such that $\left\|u_\ell^j\right\|_{H^1(\Omega_0)}\leq C_0^\ell$ uniformly for $\ell=0,1,2,\dots$. 
When we approximate $u_{j,k^2}^+$ by interpolation,  from standard error estimation of interpolation,
\[\left\|{u}_{j,k_0^2,N,M,h}^+-{u}_{j,k_0^2}^+\right\|_{H^1(\Omega_0)}\leq C_0^{M} |K|^{M},\]
where $|K|=\max_{x\neq y\in K}|x-y|$.  Thus we can finally obtain the error estimation of the algorithm:
\begin{equation}
 \label{eq:err_special}
 \left\|{u}_{k^2_0,N,M,h}-{u}_{k^2_0}\right\|_{H^\ell(\Omega_0)}\leq C\left(N^{-N_0+1/2}+h^{2-\ell}+C_0^M|K|^{M}\right)\|f\|_{L^2(\Omega_0)},
\end{equation}
where $\ell=0,1$. 
However, in numerical results, the convergence rate of the third term is difficult to analyze. On one hand, to make sure that the Taylor expansion converges, we require that $|K|$ is sufficiently small; on the other hand, the error $\left\|u_{j,k^2,N,\high{M},h}^+-u_{j,k^2}^+\right\|_{H^\ell(\Omega_0)}$ becomes larger when $k^2\rightarrow k_0^2$ due to the pole at $k_0^2$, as $C=C(|K|)$ may blow up. Moreover, $C_0$ can be a large number and it is impossible to be evaluated. 

\begin{remark}
 Compared to the method introduced in \cite{Ehrhardt2009a}, our method has some advantages. First, as the function $u_{k^2}^j$ depends analytically on $k^2\in K$, an interpolation technique can be applied to the approximation, which is better than extrapolation methods. Second, the analytical dependence is proved, but the Taylor's expansion with respect to $\epsilon$ is an assumption without proof. However, \high{both methods} have a  disadvantage in common, that is when $k^2\rightarrow k_0^2$ or $\epsilon\rightarrow 0^+$, the error brought by numerical schemes becomes larger, so the error estimation can not be explicitly described.
\end{remark}

\subsection{PM method}

As was shown, the representation \eqref{eq:LAP_sep} still holds when $S_+^0\cap S_-^0\neq\emptyset$. We extend the PM method to this case in this subsection. Similar to  Section \ref{sec:pm}, we can still find out all the $\alpha_j^\pm$ by solving \eqref{eq:gep}. However, when $S_+^0\cap S_-^0\neq\emptyset$, there may be some $\alpha_j^\pm$ such that there are corresponding eigenfunctions propagating to different directions. To this end, we introduce  an energy criteria to decide the direct of a
propagating mode (see \cite{Ehrhardt2009a}). First, we introduce an energy flux:
\[
\mathcal{E}(\phi,\phi)=4k\Im\int_{\Gamma_j}\frac{\partial\phi}{\partial x_1}\overline{\phi}\d s.
\]
When $\phi$ is propagating to the right, $\mathcal{E}(\phi; \phi) > 0$; while when it is propagating
to the left, $\mathcal{E}(\phi; \phi) < 0$. Suppose we have already found out all the elements in $P_+\bigcup P_-$, i.e., $\alpha_1,\dots,\alpha_Q$. For any $j=1,\dots,Q$, there are $L_j$ corresponding eigenfunctions $\psi_{j,\ell}(\alpha_j^+,\cdot)$. Then \eqref{eq:LAP_sep} is rewritten as:
\begin{equation*}
u(x_1+n,x_2)=\begin{cases}
\displaystyle
\begin{aligned}&  \i\left[\sum_{j=1}^Q\sum_{\ell=1}^{L_j}\mathds{1}_{j,\ell}
\frac{\exp(\i n\alpha_j)\left<f,\psi_{j,\ell}(\alpha_j,\cdot)\right>}{(\mu_{j,\ell})'(\alpha_j)}\psi_{j,\ell}(\alpha_j,\cdot)\right]      \\&\,+
\frac{1}{2\pi\i}\oint_{|z|=\exp(-\tau)} w(z,x)z^{n-1}\d z,\qquad   \qquad\qquad\qquad n\geq 0;
\end{aligned}\\
\\
\displaystyle
\begin{aligned}&  \i\left[\sum_{j=1}^Q\sum_{\ell=1}^{L_j}\left(1-\mathds{1}_{j,\ell}\right)
\frac{\exp(\i n\alpha_j)\left<f,\psi_{j,\ell}(\alpha_j,\cdot)\right>}{(\mu_{j,\ell})'(\alpha_j)}\psi_{j,\ell}(\alpha_j,\cdot)\right]   \\&\,+\frac{1}{2\pi\i}\oint_{|z|=\exp(\tau)} w(z,x)z^{n-1}\d z,\qquad \qquad\qquad\qquad \quad  n< 0;\end{aligned}
\end{cases}
\end{equation*}
where $\mathds{1}_{j,\ell}$ is an indicator function defined by:
\[
\mathds{1}_{j,\ell}=\begin{cases}
1,\text{ if } \mathcal{E}(\psi_{j,\ell},\psi_{j,\ell})>0;\\
0, \text{ if } \mathcal{E}(\psi_{j,\ell},\psi_{j,\ell})<0.
\end{cases}
\]

\section{Numerical results}
\label{sec:numer-res}

In this section, we present some numerical examples to show the efficiency of our numerical methods, i.e., CCI-method and PM-method, for both full- and half-guide problems. For all the examples, the function $q$ is chosen as follows:
\begin{equation*}
 q(x)=\begin{cases}
       1,\quad |x-a_0|>0.3;\\
       9,\quad 0.1<|x-a_0|<0.3;\\
       1+8\zeta(|x-a_0|),\quad\text{otherwise.}
      \end{cases}
\end{equation*}
The point $a_0=(0,0.5)$, and $\zeta(t)$ is a $C^8$-continuous function defined \high{by}
\begin{equation*}
 \zeta(t)=\begin{cases}
           1,\quad t\leq a;\\
           0,\quad t\geq b;\\
           1-\left[\int_{\tau=a}^b (\tau-a)^4(\tau-b)^4\d \tau\right]^{-1}\left[\int_{\tau=a}^t (\tau-a)^4(\tau-b)^4\d \tau\right],\, a<t<b
          \end{cases}
\end{equation*}
with $a=0.1$, $b=0.3$.

\begin{figure}[ht]
\centering
\includegraphics[width=12cm]{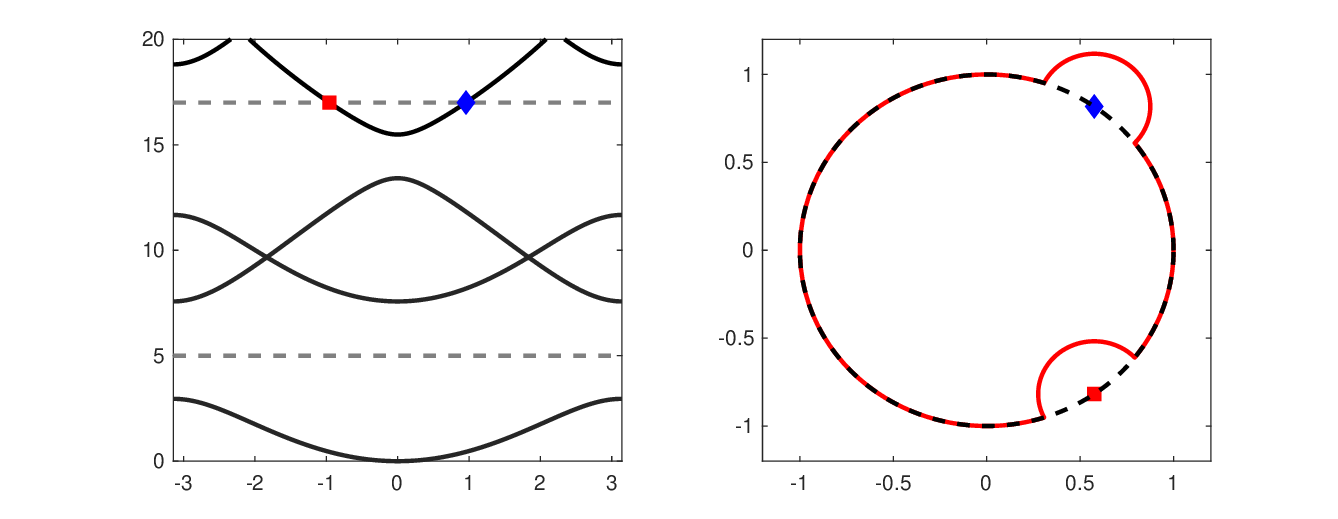}
\caption{Left: dispersion diagram; right: $z$-space.}
\label{fig:disp}
\end{figure}

First, we draw the dispersion diagram and the corresponding $z$-space, see Figure \ref{fig:disp}. From the dispersion diagram, we find out two stop bands, i.e., $(2.956\pm0.01,7.574\pm0.01)$ and $(13.41\pm0.01,15.49\pm0.01)$. When $k^2=5$, it lies in the first stop band, thus there is no eigenvalue on $\UC$ in the $z$-space, i.e., $Q=0$. In this case, $u$ is represented by \eqref{eq:u_inv} with the integral on the unit circle $\UC$. However, when $k^2=17$, there are two points lying on the dispersion curve with $\mu_n(\alpha)=17$. This implies that $Q=1$ and $\alpha_1^+=0.9576(\pm 0.01)$ (blue diamond) and $\alpha_1^-=-0.9576(\pm 0.01)$ (red square). Thus we design the integral curve in \eqref{eq:u_LAP} as the red curve on the right. Moreover, we also check the condition number of the matrix obtained from the finite element discretization of \eqref{eq:quasi_periodic_var}, to find out a rough location of poles of the function $w(z,\cdot)$ (see Figure \ref{fig:cn}).
The parameter $N_0$ is fixed \high{to be} $6$ for all the numerical examples in this section.

\begin{figure}[ht]
\centering
\includegraphics[width=6cm]{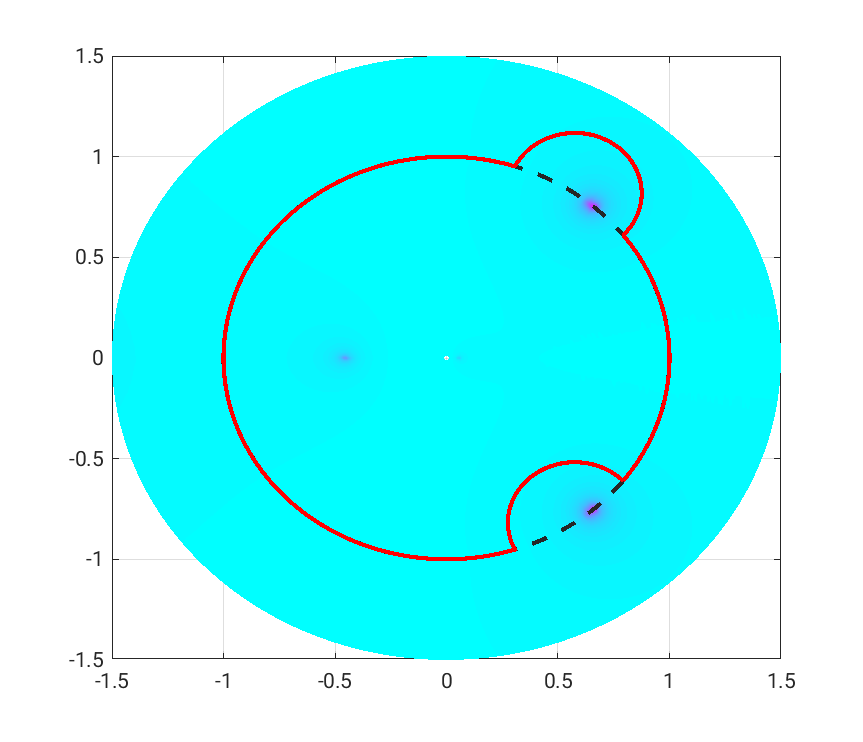}
\caption{Condition number on the complex \high{plane}.}
\label{fig:cn}
\end{figure}

\subsection{Full-guide problems}
\label{sec:full-guide-numer}

In the first part of this section, we show some numerical examples for scattering problems in full-waveguide, when Assumption \ref{asp1} and \ref{asp2} are both satisfied. The compactly supported source term $f$ is defined as follows:
\begin{equation*}
f(x)=\begin{cases}
      0,\quad |x-a_0|>0.3;\\
      3\cos(2\pi x_1)\sin(2\pi x_2),\quad 0.1<|x-a_0|<0.3;\\
      3\zeta(|x-a_0|)\cos(2\pi x_1)\sin(2\pi x_2),\quad\text{otherwise;}
     \end{cases}
\end{equation*}
where $a_0=(0,0.5)^\top$.

With these \high{data}, we calculate the value of $u$ \high{for } different parameters.  For the finite element method, we choose $h=0.0025, 0.005, 0.01,0.02$ and $N=8,16,32,64$ for $k^2=5$, $N=4,8,16,32,64$
for $k^2=17$. We also compute ``exact solutions'' from the finite element method introduced in \cite{Ehrhardt2009a,Ehrhardt2009} with $h=0.005$ and the Lagrangian element. First we show the relative errors with different $h$ and $N$ for both cases, defined by
\begin{equation*}
 err_{N,h}=\frac{\|u_{N,h}-u_{exa}\|_{L^2(\Omega_0)}}{\|u_{exa}\|_{L^2(\Omega_0)}},
\end{equation*}
where $u_{N,h}$ is the numerical solution with parameter $N$ and $h$, and $u_{exa}$ is the ``exact solution''. Note that for $k^2=5$, the CCI-method and PM-method are the same, and the results are shown in Table \ref{k5}. For $k^2=17$, the results from the CCI-method are shown in Table \ref{k17} and from the PM-method are in \ref{k17-pm}. We can see that the relative error decays as $N$ gets larger and $h$ gets smaller. Note that when $N$ is large enough (e.g., $N\geq 32$), the relative error does not decay when $N$ gets larger, this implies that the error brought by $N$ is relatively smaller, compared with the error from $h$. The decay rate of the CCI-method and the PM-method are relevant. 

\begin{table}[htb]
\caption{Relative $L^2$-errors for $k^2=5$.}
\label{k5}       
\begin{tabular}{lllll}
\hline\noalign{\smallskip}
 & $h=0.02$ & $h=0.01$ & $h=0.005$ & $h=0.0025$ \\
\noalign{\smallskip}\hline\noalign{\smallskip}
$N=8$&$2.91$E$-02$&$2.78$E$-02$&$2.78$E$-02$&$2.73$E$-02$\\
$N=16$&$1.99$E$-03$&$6.63$E$-04$&$3.27$E$-04$&$2.49$E$-04$\\
$N=32$&$1.79$E$-03$&$4.53$E$-04$&$1.06$E$-04$&$3.51$E$-05$\\
$N=64$&$1.79$E$-03$&$4.53$E$-04$&$1.06$E$-04$&$3.51$E$-05$\\
\noalign{\smallskip}\hline
\end{tabular}
\end{table}

\begin{table}[htb]
\caption{Relative $L^2$-errors for $k^2=17$-CCI-method.}
\label{k17}       
\begin{tabular}{lllll}
\hline\noalign{\smallskip}
  & $h=0.02$ & $h=0.01$ & $h=0.005$ & $h=0.0025$ \\
\noalign{\smallskip}\hline\noalign{\smallskip}
$N=4$&$1.05$E$-01$&$1.06$E$-01$&$1.07$E$-01$&$1.07$E$-01$\\
$N=8$&$1.99$E$-02$&$2.22$E$-02$&$2.29$E$-02$&$2.30$E$-02$\\
$N=16$&$4.67$E$-03$&$9.48$E$-04$&$6.27$E$-04$&$7.56$E$-04$\\
$N=32$&$5.24$E$-03$&$1.34$E$-03$&$3.58$E$-04$&$1.44$E$-04$\\
$N=64$&$5.24$E$-03$&$1.34$E$-03$&$3.59$E$-04$&$1.45$E$-04$\\
\noalign{\smallskip}\hline
\end{tabular}
\end{table}

\begin{table}[htb]
	\caption{Relative $L^2$-errors for $k^2=17$-PM-method.}
	\label{k17-pm}       
	\begin{tabular}{lllll}
		\hline\noalign{\smallskip}
		& $h=0.02$ & $h=0.01$ & $h=0.005$ & $h=0.0025$ \\
		\noalign{\smallskip}\hline\noalign{\smallskip}
		$N=4$&$1.23$E$-01$&$1.24$E$-01$&$1.23$E$-01$&$1.24$E$-01$\\
		$N=8$&$9.83$E$-03$&$8.62$E$-03$&$8.51$E$-03$&$8.50$E$-03$\\
		$N=16$&$4.79$E$-03$&$1.23$E$-03$&$3.28$E$-04$&$1.38$E$-04$\\
		$N=32$&$4.79$E$-03$&$1.23$E$-03$&$3.26$E$-04$&$1.36$E$-04$\\
		$N=64$&$4.79$E$-03$&$1.23$E$-03$&$3.26$E$-04$&$1.36$E$-04$\\
		\noalign{\smallskip}\hline
	\end{tabular}
\end{table}

As the convergence rate with respect to $h$ is classical, we are especially interested in that \high{of} $N$. We fix $h=0.01$ for both cases, and compute the relative error between $u_{N,h}$ and $u_{256,h}$ for $k^2=5$, $u_{128,h}$ for $k^2=17$. From the result in \eqref{eq:err_quad}, the error is expected to decay at the rate of $O(N^{-5.5})$. From the two pictures in Figure \ref{fig:err}, \high{the convergence is} even faster than expected.

\begin{figure}[h]
\centering
\begin{tabular}{c  c}
\includegraphics[width=0.45\textwidth]{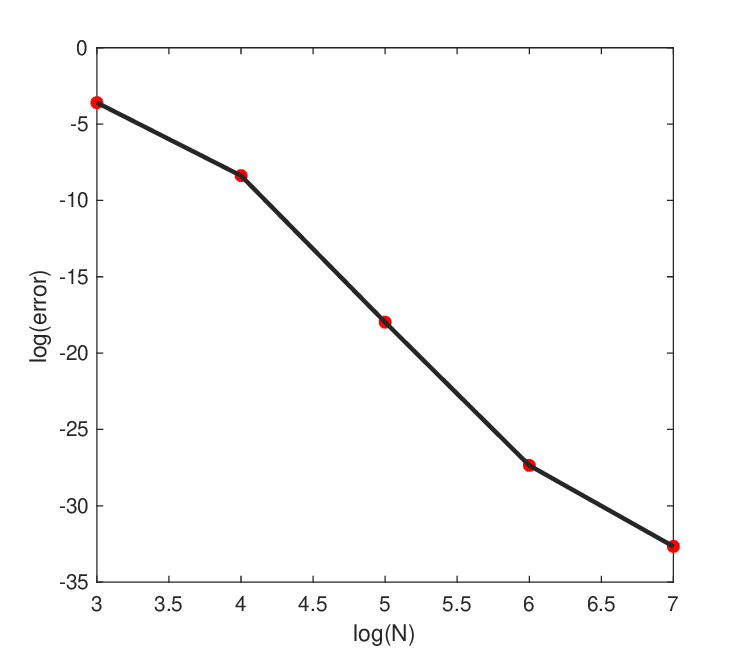} 
& \includegraphics[width=0.45\textwidth]{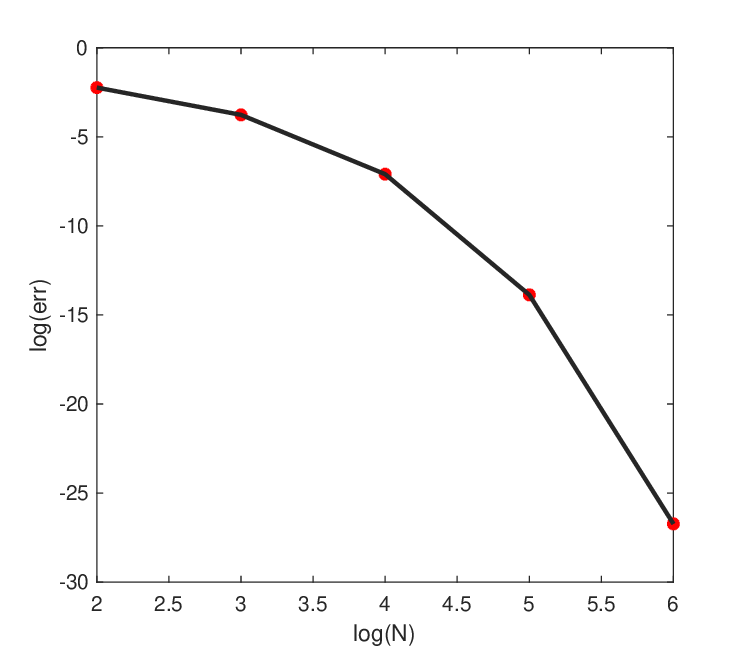}\\[-0cm]
\end{tabular}
\caption{Left: $k^2=5$; Right: $k^2=17$.}
\label{fig:err}
\end{figure}

We also compute the energy fluxes of propagating Bloch wave solutions for this example. We approximate the integrals
\[
 u_1^\pm=\frac{1}{2\pi\i}\int_{|z-\exp(\i\alpha_1^\pm)|=r}w(z,x)z^{-1}\d z
\]
with the same method, and evaluate the energy fluxes
\[
 \mathcal{E}(u_1^\pm,u_1^\pm)=4k\,\Im\left(\int_{\Gamma_1}\frac{\partial u_1^\pm}{\partial x_1}\overline{u_1^\pm}d s\right).
\]
Fix parameters $r=0.1$ and $N=16$, and the mesh size $h=0.005$. We obtain the \high{values}
\[
 \mathcal{E}(u_1^+,u_1^+)\approx 6.362\times 10^{-11},\quad \mathcal{E}(u_1^-,u_1^-)\approx -6.362\times 10^{-11}.
\]
This shows that $u_1^+$ is propagating to the right, and $u_1^-$ is propagating to the left according to the energy criteria. This also coincides with the analysis in Section 6.3.

\subsection{Half-guide problems}
\label{sec:half-guide-numer}
In the second part of this section, we show some numerical examples for half guide problems. The boundary data is given by
\begin{equation*}
 \phi(x_1,x_2)=\sin x_2+\frac{x_2^2}{2}+\exp(\i x_1)\cos(2 x_2).
\end{equation*}
For all the examples, we approximate the source term by
\begin{equation*}
 f_\gamma(x_1,x_2)\approx \sum_{j=-20}^{20}\sum_{\ell=0}^{10}\widehat{f}_{j,\ell}(\gamma)\exp(2\i\pi j x_1)\cos(\pi \ell x_2),
\end{equation*}
where $\gamma$ is the regularization parameter.
As the numerical results also depend on $\gamma$, we show different results with respect to different regularization parameters, i.e., $\gamma=10^{-2}$ and $10^{-5}$. The ``exact solutions'' also come from the method introduced in \cite{Ehrhardt2009,Ehrhardt2009a}, and the relative error $err_{N,h}$ is  defined in the same way. For both cases, the numerical solutions are computed by the CCI-method. We show the results for $k^2=5$ with parameters $N=8,16,32$ and $h=0.02,0.01,0.005$ with $\gamma=10^{-2}$ in Table \ref{k5_half_1} and with $\gamma=10^{-5}$ in Table \ref{k5_half_2}. For $k^2=17$, we show results for $N=4,8,16$ and $h=0.02,0.01,0.005$ with $\gamma=10^{-2}$ in Table \ref{k17_half_1}, and with $\gamma=10^{-5}$ in Table \ref{k17_half_2}. We also give the contour map for the solution with $k^2=17$, $\gamma=10^{-5}$, $N=16$ and $h=0.05$ in Figure \ref{fig:solution}. For all these cases, we can see that the error decays when $N$ gets larger (especially when $h$ is small) and $h$ gets smaller (especially when $N$ is large). However, the decaying rate slows down significantly when the parameter $N$ becomes sufficiently large (e.g., $N\geq 16$). This comes from the cut-off approximation of the series of $f$ and the regularization process. We also notice that the relative errors \high{corresponding} to $\gamma=10^{-2}$ is larger than that to $\gamma=10^{-5}$, which is also as expected.

\begin{figure}[h]
\centering
\begin{tabular}{c  c}
\includegraphics[width=0.45\textwidth]{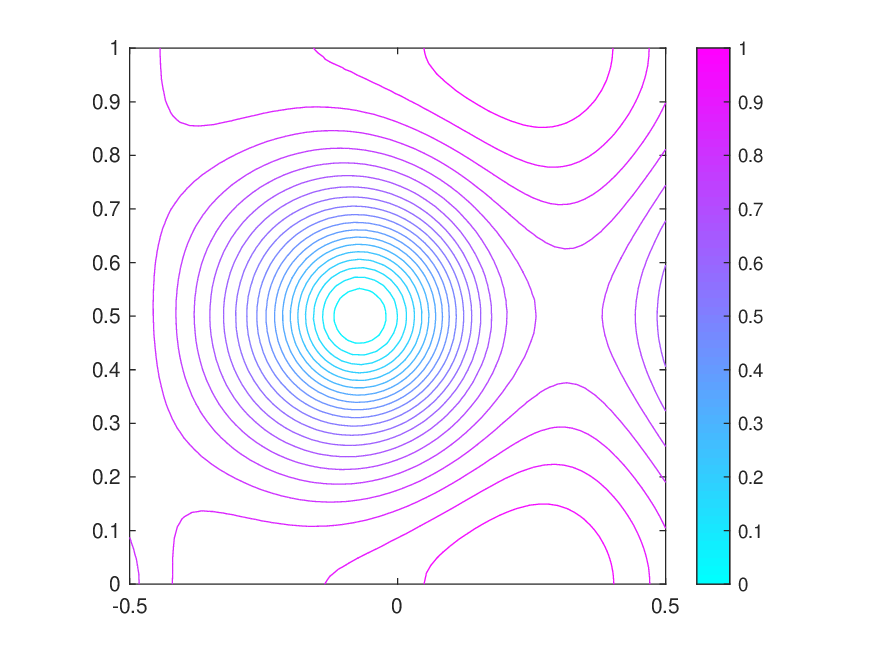} 
& \includegraphics[width=0.45\textwidth]{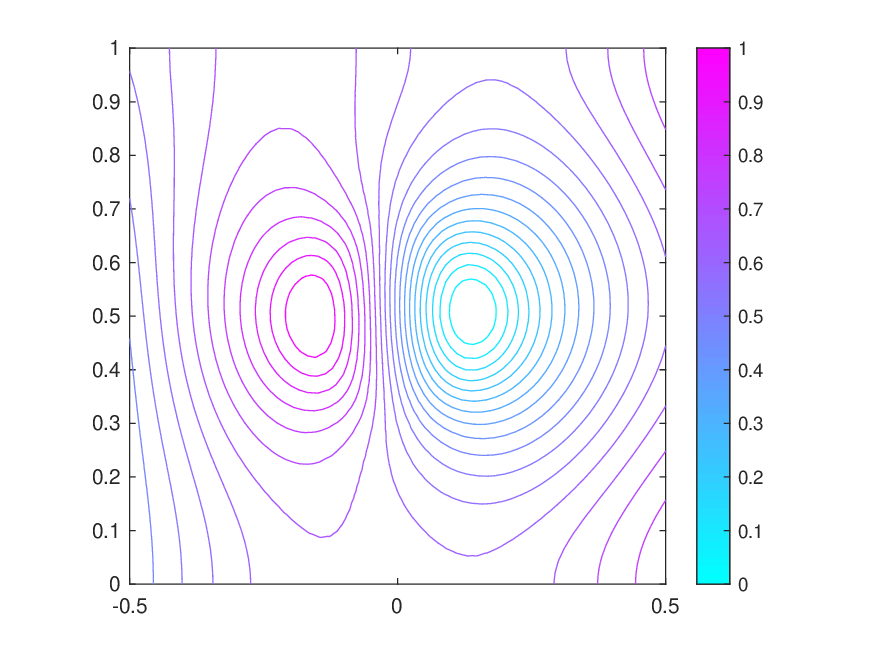}\\[-0cm]
\end{tabular}
\caption{Contour map of the solution with $k^2=17$. Left: real part of the solution; Right: imaginary part of the solution.}
\label{fig:solution}
\end{figure}

\begin{table}[htb]
\caption{Relative $L^2$-errors for $k^2=5$, $\gamma=10^{-2}$.}
\label{k5_half_1}       
\begin{tabular}{llll}
\hline\noalign{\smallskip}
 & $h=0.02$ & $h=0.01$ & $h=0.005$\\
\noalign{\smallskip}\hline\noalign{\smallskip}
$N=8$&$4.24$E$-02$&$3.96$E$-02$&$3.86$E$-02$\\
$N=16$&$2.11$E$-02$&$1.25$E$-02$&$7.69$E$-03$\\
$N=32$&$2.10$E$-02$&$1.25$E$-02$&$7.66$E$-03$\\
\noalign{\smallskip}\hline
\end{tabular}
\end{table}

\begin{table}[htb]
\caption{Relative $L^2$-errors for $k^2=5$, $\gamma=10^{-5}$.}
\label{k5_half_2}       
\begin{tabular}{llll}
\hline\noalign{\smallskip}
 & $h=0.02$ & $h=0.01$ & $h=0.005$\\
\noalign{\smallskip}\hline\noalign{\smallskip}
$N=8$&$3.79$E$-02$&$3.80$E$-02$&$3.80$E$-02$\\
$N=16$&$3.38$E$-03$&$9.53$E$-04$&$4.49$E$-04$\\
$N=32$&$3.34$E$-03$&$8.76$E$-04$&$3.01$E$-04$\\
\noalign{\smallskip}\hline
\end{tabular}
\end{table}

\begin{table}[htb]
\caption{Relative $L^2$-errors for $k^2=17$, $\gamma=10^{-2}$.}
\label{k17_half_1}       
\begin{tabular}{llll}
\hline\noalign{\smallskip}
 & $h=0.02$ & $h=0.01$ & $h=0.005$\\
\noalign{\smallskip}\hline\noalign{\smallskip}
$N=4$&$5.14$E$-02$&$4.67$E$-02$&$4.94$E$-02$\\
$N=8$&$4.11$E$-02$&$2.09$E$-02$&$1.17$E$-02$\\
$N=16$&$4.04$E$-02$&$1.96$E$-02$&$9.42$E$-03$\\
\noalign{\smallskip}\hline
\end{tabular}
\end{table}

\begin{table}[htb]
\caption{Relative $L^2$-errors for $k^2=17$, $\gamma=10^{-5}$.}
\label{k17_half_2}       
\begin{tabular}{llll}
\hline\noalign{\smallskip}
 & $h=0.02$ & $h=0.01$ & $h=0.005$\\
\noalign{\smallskip}\hline\noalign{\smallskip}
$N=4$&$6.18$E$-02$&$5.68$E$-02$&$5.58$E$-02$\\
$N=8$&$1.70$E$-02$&$7.97$E$-03$&$7.19$E$-03$\\
$N=16$&$1.55$E$-02$&$3.81$E$-03$&$1.93$E$-03$\\
\noalign{\smallskip}\hline
\end{tabular}
\end{table}

\subsection{Special wavenumbers}
\label{sec:special-numer}

In this section, we show some numerical results when Assumption \ref{asp2} is not \high{satisfied}, i.e., \high{$S_+\cap S_-\neq\emptyset$}.
From the dispersion diagram, i.e., Figure \ref{fig:disp0}, when $k^2=9.6663(\pm 0.01)$, Assumption \ref{asp2} is not satisfied. Thus the method introduced in Section \ref{sec:special} is used for the numerical simulation. From the dispersion curve, $Q'=Q=2$,  \high{$\alpha_1^+=\alpha_1^-=-1.8333(\pm 0.01)$ and $\alpha_2^+=\alpha_2^-=1.8333(\pm 0.01)$}. Let the curve $\mathcal{C}_0:=\big\{z\in\C:\,|z|=0.8\big\}$, and $B_j^\pm$ ($j=1,2$) be defined by \eqref{eq:ball} with $\delta_0=0.1$. For the visualization of the points and curves we refer to Figure \ref{fig:disp0}.

\begin{figure}[ht]
\centering
\includegraphics[width=14cm,height=5.5cm]{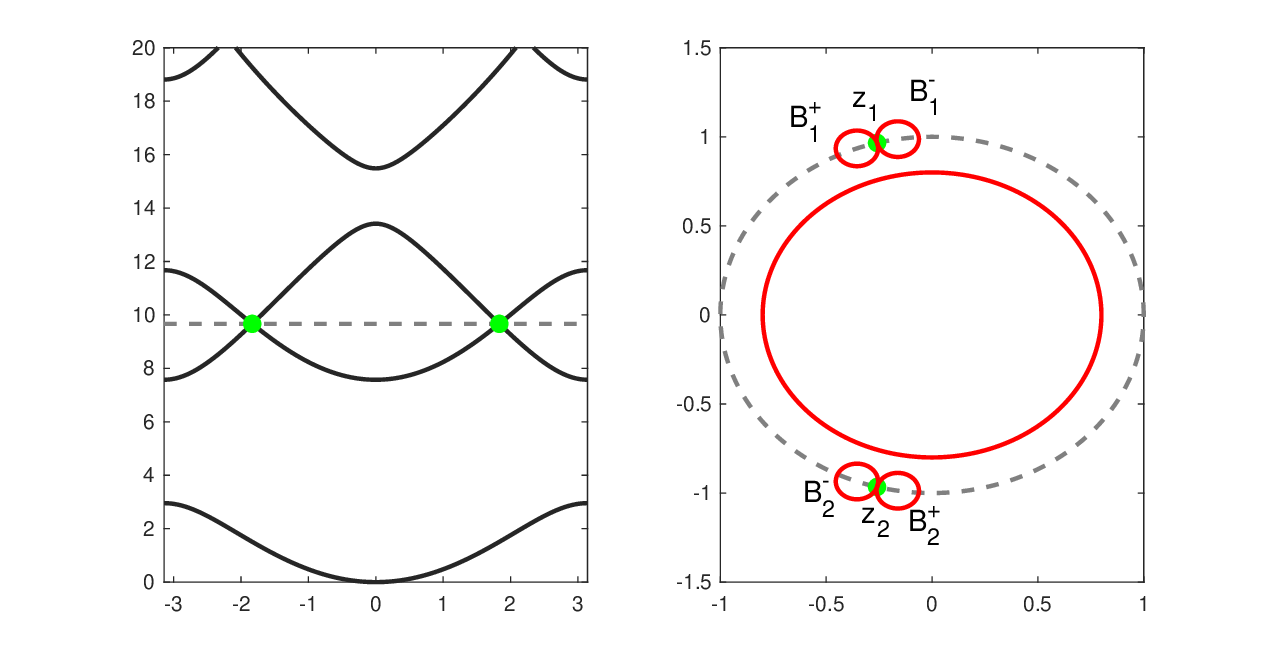}
\caption{$k^2=9.6663(\pm 0.01)$. Left: dispersion diagram ; right: $z$-space.}
\label{fig:disp0}
\end{figure}

For the CCI-method, we choose two different interpolation strategies to carry out the numerical approximation. For the first strategy, set $M=3$, and 
\[k_1^2=k^2-0.1,\,k_2^2=k^2+0.1,\,k_3^2=k^2+0.2.\]
For the second one, $M=5$ and
\[k_1^2=k^2-0.2,\,k_2^2=k^2-0.1,\,k_3^2=k^2+0.1,\,k_4^2=k^2+0.2,\,k_5^2=k^2+0.3.\]
We still use the result obtained by the method introduced in \cite{Ehrhardt2009,Ehrhardt2009a} to produce the ``exact solution'', and compute the relative errors with the parameters $h=0.02,0.01,0.005,0.0025$ and $N=16,32,64,128$. The results are shown in Table \ref{k_special_2}-\ref{k_special_4}. In both tables, the relative errors with these two different strategies are similar, and the error decays when $N$ gets larger and $h$ gets smaller. These results show that the CCI-method for special numbers is convergent.

\high{
We also used the PM-method to approximate the LAP solution with this special wavenumber. We adopt the same parameters and the results are shown in Table \ref{k_special_pm}. From the results, the relative errors decay as $N$ gets larger and $h$ gets smaller. For fixed $h$, the decay becomes slower when $N>64$; for fixed $N$, the error even becomes larger when $h<0.5$. This may come from the error of the eigs function from Matlab. 
}

\begin{table}[htb]
\caption{Relative $L^2$-errors for $k^2=9.6663(\pm 0.01)$: CCI-method-strategy 1.}
\label{k_special_2}       
\begin{tabular}{lllll}
\hline\noalign{\smallskip}
  & $h=0.02$ & $h=0.01$ & $h=0.005$ & $h=0.0025$\\
\noalign{\smallskip}\hline\noalign{\smallskip}
$N=16$&$3.02$E$-01$&$3.05$E$-01$&$3.06$E$-01$&$3.06$E$-01$\\
$N=32$&$3.46$E$-02$&$3.40$E$-02$&$3.38$E$-02$&$3.37$E$-02$\\
$N=64$&$3.61$E$-03$&$1.81$E$-03$&$1.55$E$-03$&$1.51$E$-03$\\
$N=128$&$2.85$E$-03$&$6.57$E$-04$&$1.74$E$-04$&$1.67$E$-04$\\
\noalign{\smallskip}\hline
\end{tabular}
\end{table}

\begin{table}[htb]
\caption{Relative $L^2$-errors for $k^2=9.6663(\pm 0.01)$: CCI-method-strategy 2.}
\label{k_special_4}       
\begin{tabular}{lllll}
\hline\noalign{\smallskip}
  & $h=0.02$ & $h=0.01$ & $h=0.005$ & $h=0.0025$\\
\noalign{\smallskip}\hline\noalign{\smallskip}
$N=16$&$2.89$E$-01$&$2.91$E$-01$&$2.92$E$-01$&$2.92$E$-01$\\
$N=32$&$3.52$E$-02$&$3.50$E$-02$&$3.48$E$-02$&$3.47$E$-02$\\
$N=64$&$3.62$E$-03$&$1.82$E$-03$&$1.56$E$-03$&$1.52$E$-03$\\
$N=128$&$2.85$E$-03$&$6.49$E$-04$&$1.61$E$-04$&$1.55$E$-04$\\
\noalign{\smallskip}\hline
\end{tabular}
\end{table}

\begin{table}[htb]
	\caption{Relative $L^2$-errors for $k^2=9.6663(\pm 0.01)$: PM-method.}
	\label{k_special_pm}       
	\begin{tabular}{lllll}
		\hline\noalign{\smallskip}
		& $h=0.02$ & $h=0.01$ & $h=0.005$ & $h=0.0025$\\
		\noalign{\smallskip}\hline\noalign{\smallskip}
		$N=16$&$3.64$E$-01$&$3.62$E$-01$&$3.62$E$-01$&$3.62$E$-01$\\
		$N=32$&$3.50$E$-02$&$3.33$E$-02$&$3.24$E$-02$&$3.25$E$-02$\\
		$N=64$&$4.83$E$-03$&$4.57$E$-03$&$1.84$E$-03$&$3.94$E$-03$\\
		$N=128$&$4.49$E$-03$&$4.30$E$-03$&$1.07$E$-03$&$3.65$E$-03$\\
		\noalign{\smallskip}\hline
	\end{tabular}
\end{table}

We also check the energy fluxes \high{corresponding} to the propagating modes. The approximation is carried out with the help of the second strategy, with parameters $h=0.005$ and $N=64$. The energy fluxes \high{corresponding} to $u_1^\pm$ and $u_2^\pm$ are evaluated as follows:
\[
\begin{aligned}
 &\mathcal{E}(u_1^+,u_1^+)\approx 3.009\times 10^{-10},\quad \mathcal{E}(u_1^-,u_1^-)\approx -3.009\times 10^{-10},\\
 & \mathcal{E}(u_2^+,u_2^+)\approx 0.0399,\quad \mathcal{E}(u_2^-,u_2^-)\approx -0.0399.
 \end{aligned}
\]
From the values, $u_1^+$ and $u_2^+$ are propagating to the right, $u_1^-$ and $u_2^-$ are propagating to the left. This coincides with the results shown in Section 6.3.

\subsection{Conclusion}

Now we compare the two different methods -- the CCI-method and the PM-method. The CCI-method depends on a simplified integral representation \eqref{eq:u_LAP} for the LAP solution with Assumption \ref{asp1}, where a suitable complex integral curve is to be designed specially depending on the behaviour of the Floquet multipliers with respect to the absorbing parameter $\epsilon$. Then the LAP solution is approximated by the sum of finite number of solutions of quasi-periodic problems, which are all well-posed. However, to know the behaviour of the poles, for example by producing a dispersion diagram, may take a relatively longer time. When Assumption \ref{asp1} no longer holds, an interpolation technique is introduced to make the CCI-method suitable for this situation. This makes the CCI-method more complicated. On the other hand, the PM-method is based on the curve integral with finite number of non-selfadjoint eigenvalue problem. Compared to the CCI-method, it does not depend on Assumption \ref{asp1}. We can decide if a Floquet mode is acceptable by the sign of the energy flux, so we do not need to know the dispersion curve in principle. However, as the non-selfadjoint eigenvalue problems have complex eigenvalues, sometimes it may not be easy to find out all real eigenvalues in $(-\pi,\pi]$. The safest way is to find out a rough guess of the eigenvalues from the dispersion diagram first, then find the eigenvalues nearest to the initial guess. Moreover, we still need to solve more eigenvalue problems to evaluate $\mu'(\alpha)$.  

We also compare the methods introduced in this paper with other methods. The computational complexity of both methods are equivalent to that of \cite{Zhang2017e}. The problems are different, for the Floquet-Bloch transformed field  has finite number of poles in this paper, but one or two branch cuts in \cite{Zhang2017e}. The methods introduced in \cite{Joly2006,Fliss2009a,Coatl2012} are based on the numerical evaluation of the DtN maps, which are described by the quadratic characteristic equation. The evaluations are carried out by the iteration method based on the cell problem, thus this may involve many times of the solutions of quasi-periodic problem. Another interesting method is introduced in \cite{Dohna2018}, by approximating the LAP solution by finite number of propagating modes and a truncated problem. Suppose the truncated problem exists in the cells from $-N$ to $N$, and the degree of freedom is $M$ in one cell, then the degree of freedom for the whole problem is greater than $NM$. Thus they have to solve a system of at least $NM\times NM$. However, for our problem we only need to solve several times of $M\times M $ linear system, which is much more efficient. Due to the super algebraic convergence, we do not need to solve the $M\times M $ linear system too many times.  Thus our method is faster than the one introduced in this paper.

\section*{Appendix}

\begin{proof}[Proof of Lemma \ref{th:accumulation}]
Assume that the set has a bounded accumulation point $k_0\in\R_+$, i.e., there is a sequence $k_n$ such that 
\begin{equation*}
P_+(k_n)\cap P_-(k_n)\neq\emptyset,\quad \lim_{n\rightarrow\infty}k_n=k_0.
\end{equation*}
Thus the sequence has a strictly monotones subsequence which also converges to $k_0$. Without loss of generality, we assume that the subsequence is monotonically decreasing and is still denoted by $k_n$, i.e.,
\begin{equation*}
k_0^2<\cdots<k_n^2<k_{n-1}^2<\cdots<k_1^2.
\end{equation*}
This implies that  for any $n\in\N$, there is a pair $(i_n\,,j_n)\in\N\times\N$ such that
\begin{equation*}
\exists\, \alpha_n\in(-\pi,\pi],\,s.t.,\,\mu_{i_n}(\alpha_n)=\mu_{j_n}(\alpha_n)=k^2_n
\end{equation*} 
that satisfies
\begin{equation*}
\mu_{i_n}'(\alpha_n)>0,\quad
\mu_{j_n}'(\alpha_n)<0.
\end{equation*}
 As $\lim_{n\rightarrow\infty}\mu_n(\alpha)=\infty$ for any $\alpha\in(-\pi,\pi]$, there should be a  subsequence of pairs $(i_n,\,j_n)$ such that $i_n=i_0$ and $j_n=j_0$ where $i_0$ and $j_0$ are two constant positive integers. Still denote the subsequence of $k_n^2$ by $k_n^2$, then for any $n\in\N$, there is an $\alpha_n\in(-\pi,\pi]$ such that
\begin{equation*}
\mu_{i_0}(\alpha_n)=\mu_{j_0}(\alpha_n)=k^2_n,\quad\mu'_{i_0}(\alpha_n)>0,\,\mu'_{j_0}(\alpha_n)<0.
\end{equation*}

Define the function
\begin{equation*}
\mu(\alpha):=\mu_{i_0}(\alpha)-\mu_{j_0}(\alpha),
\end{equation*}
then there is a sequence $\alpha_n\in(-\pi,\pi]$ such that
\begin{equation*}
\mu(\alpha_n)=0,\quad\forall n\in\N.
\end{equation*}
As $\mu_{i_0}$ and $\mu_{j_0}$ are both analytic functions, $\mu$ is analytic as well. Thus either $\mu$ is a constant function equals to $0$, or  $\alpha_n=\alpha_0$ except for a finite number of $n$'s.

For the first case, $\mu'(\alpha)=0$ for any $\alpha_n$, which contradicts with 
\begin{equation*}
\mu'(\alpha_n)=\mu'_{i_0}(\alpha_n)-\mu'_{j_0}(\alpha_n)>0.
\end{equation*}
 For the second case, suppose there is an $N>>1$ such that $\alpha_n=\alpha_0$ for any $n\geq N$, then $\mu_{i_0}(\alpha_n)=\mu_{j_0}(\alpha_n)=k^2_n$ implies that $k^2_n=k^2_0$ for any $n\geq N$. This contradicts with the monotone decreasing property. Thus $k_0^2$ can not be an accumulation point, the proof is finished.

\end{proof}

\section*{Acknowledgments}
The work  is funded by the Deutsche Forschungsgemeinschaft (DFG, German Research Foundation) – Project-ID 258734477 – SFB 1173


%
%



\end{document}